\documentclass[a4paper,reqno, 12pt]{amsart}

\usepackage{amsmath}
\usepackage{paralist}
\usepackage{graphics}
\usepackage{epsfig}
\usepackage{amsfonts}
\usepackage{amssymb}
\usepackage{color}
\usepackage{epstopdf}

\newtheorem{theorem}{Theorem}[section]
\newtheorem{lemma}[theorem]{Lemma}

\def\disp{\displaystyle}

\def\ifl{\iffalse }

\def\bc{\begin{center}}       \def\ec{\end{center}}
\def\ba{\begin{array}}        \def\ea{\end{array}}
\def\be{\begin{equation}}     \def\ee{\end{equation}}
\def\bea{\begin{eqnarray}}    \def\eea{\end{eqnarray}}
\def\beaa{\begin{eqnarray*}}  \def\eeaa{\end{eqnarray*}}
\def\dis{\displaystyle}

\numberwithin{equation}{section}
\newtheorem{proposition}[theorem]{Proposition}

\newtheorem{remark}[theorem]{Remark}
\numberwithin{equation}{section}

\begin{document}
\author{Huicong Li}
\address{School of Mathematics, Sun Yat-sen University, Guangzhou, 510275, Guangdong Province, China}
\email{lihuicong@mail.sysu.edu.cn}
\author{Rui Peng}
\address{School of Mathematics and Statistics, Jiangsu Normal University, Xuzhou, 221116, Jiangsu Province, China}
\email{pengrui\_seu@163.com}
\author{Tian Xiang$^\dagger$}
\address{Institute for Mathematical Sciences, Renmin University of China, Bejing, 100872, China}
\email{txiang@ruc.edu.cn}

\title[Dynamics and asymptotics of steady states]{Dynamics and asymptotic profiles of endemic equilibrium for two frequency-dependent SIS epidemic models with
cross-diffusion$^*$}
\thanks{$^\dagger$ Corresponding author} 
\thanks{$^*$H. Li was partially supported by China Postdoctoral Science Foundation (No. 2016M590335)
and NSF of China (Nos. 11701180, 11671143 and 11671144), R. Peng was partially supported by NSF of China (Nos. 11671175 and 11571200), the Priority Academic Program Development of Jiangsu Higher Education Institutions, Top-notch Academic Programs Project of Jiangsu Higher Education Institutions (No. PPZY2015A013) and Qing Lan Project of Jiangsu Province, and T. Xiang was partially supported by NSF of China (Nos. 11601516 and 11571363) and  the Research Funds  of Renmin University of China (No. 2018030199).}

\begin{abstract}
This paper is concerned with two frequency-dependent SIS epidemic reaction-diffusion models in heterogeneous environment, with a cross-diffusion term modeling the effect that susceptible individuals tend to move away from higher concentration of infected individuals.
It is first shown that the corresponding Neumann initial-boundary value problem in an $n$-dimensional bounded smooth domain  possesses a  unique global classical solution  which is uniformly-in-time bounded regardless of the strength of the cross-diffusion and the spatial dimension $n$. It is further shown that, even in the presence of cross-diffusion,  the models still admit threshold-type dynamics in terms of the basic reproduction number $\mathcal R_0$; that is, the unique disease free equilibrium is globally stable if $\mathcal R_0<1$, while if $\mathcal R_0>1$, the disease is uniformly persistent and there is an endemic equilibrium,  which is globally stable in some special cases with weak chemotactic sensitivity.
Our results on the asymptotic profiles of endemic equilibrium illustrate that restricting the motility of susceptible population may eliminate the infectious disease entirely for the first model with constant total population but fails for the second model with varying total population. In particular, this implies that such cross-diffusion does not contribute to the elimination of the infectious disease modelled by the second one.

\end{abstract}



\maketitle

\numberwithin{equation}{section}

\section{Introduction}
In this paper, we are interested in the following two diffusive SIS epidemic models with cross-diffusion and frequency-dependence:
\begin{equation}\label{SIS-1}
\begin{cases}
\disp S_t=d_S \Delta S+\chi\nabla\cdot( S \nabla I)-\beta(x) \frac{SI}{S+I}+\gamma(x) I,& x\in\Omega, t>0,\\
\disp I_t=d_I\Delta I+\beta(x) \frac{SI}{S+I}-\gamma(x) I,& x\in\Omega, t>0,\\
\disp \frac{\partial S}{\partial \nu}=\frac{\partial I}{\partial \nu}=0, & x\in \partial \Omega, t>0,\\
\disp (S(x,0),  I(x,0))=(S_0(x), I_0(x)), & x\in \Omega
\end{cases}
\end{equation}
and
\begin{equation}
\label{SIS-2}
\begin{cases}
\disp S_t=d_S \Delta S+\chi\nabla\cdot( S \nabla I)+\Lambda(x)-S-\beta(x) \frac{SI}{S+I}+\gamma(x) I,& x\in\Omega, t>0,\\
\disp I_t=d_I\Delta I+\beta(x) \frac{SI}{S+I}-\gamma(x) I,& x\in\Omega, t>0,\\
\disp \frac{\partial S}{\partial \nu}=\frac{\partial I}{\partial \nu}=0, & x\in \partial \Omega, t>0,\\
\disp (S(x,0),  I(x,0))=(S_0(x), I_0(x)), & x\in \Omega.
\end{cases}
\end{equation}
Here, $\Omega\subset\mathbb{R}^n~(n\geq 1)$ is a bounded domain with smooth boundary $\partial\Omega$. The unknown functions
$S(x,t)$ and $I(x,t)$, respectively, denote the population density of susceptible and infected individuals at location $x$ and time $t$; $d_S$ and $d_I$ are positive constants measuring the random mobility of susceptible and infected populations respectively; the cross-diffusion term $\chi \nabla \cdot (S\nabla I)$ stands for the \lq\lq chemotaxis\rq\rq\,effect that susceptible individuals are \lq\lq smart" and they tend to move away from higher concentration of infected individuals  with the positive constant $\chi$ representing the magnitude of this effect; and $\beta(x)$ and $\gamma(x)$ are positive H\"{o}lder continuous functions on $\overline\Omega$ accounting for the rates of disease transmission and recovery at location $x$, respectively. In \eqref{SIS-2}, the $S$-equation indicates that the susceptible population is subject to linear source $\Lambda-S$ with $\Lambda$ being a positive H\"{o}lder continuous function.
The homogeneous Neumann boundary conditions mean there is no population flux crossing the boundary $\partial\Omega$. As for the initial data $(S_0,I_0)$, we assume throughout this paper  that
\be
0\leq S_0\in C(\overline\Omega),\ \ \ I_0\in W^{1,\infty}(\Omega)\ \mbox{ and }\ I_0\geq0,\not\equiv0.
\label{initial}
\ee
Let
$$
N:=\int_{\Omega}(S_0(x)+I_0(x))dx>0
$$
be the total number of individuals in $\Omega$ at the initial time $t=0$. By integrating both equations in \eqref{SIS-1} and then adding the resulting identities, one can easily see that the total population is conserved. That is,
\be
\int_\Omega \left(S(x,t)+I(x,t)\right)dx=N,\quad \forall t>0.
\label{L1-bdd}
\ee
Throughout the text, we assume that $N$ is a given positive constant.
Obviously, such conservation property no longer holds for system \eqref{SIS-2}.

To investigate the effects of environmental heterogeneity and individual motility, Allen et al. \cite{Allen} proposed the following frequency-dependent SIS
(susceptible-infected-susceptible) epidemic reaction-diffusion system.
 \begin{equation}
 \left\{ \begin{array}{llll}
 \dis \frac{\partial S}{\partial t}-d_S\Delta S=-\beta(x)\frac{SI}{S+I}+\gamma(x)I,&x\in\Omega,\,t>0,\\
 \dis \frac{\partial I}{\partial t}-d_I\Delta I=\beta(x)\frac{SI}{S+I}-\gamma(x)I,&x\in\Omega,\,t>0,\\
 \dis \frac{\partial S}{\partial \nu}=\frac{\partial I}{\partial \nu}=0,&x\in\partial\Omega,\,t>0,\\
 \dis S(x,0)=S_0(x)\geq0,\,I(x,0)= I_0(x)\geq,\not\equiv 0,&x\in\Omega.
 \end{array}\right.
 \label{AllenSIS}
 \end{equation}
In \cite{Allen}, the authors defined the basic reproduction number $\mathcal R_0$ via a variational characterization and it was shown that the unique
disease-free equilibrium (DFE) is globally asymptotically stable if $\mathcal R_0<1$, whereas there exists a unique endemic equilibrium (EE) if $\mathcal R_0>1$.
Here, a DFE $(S,I)$ is an equilibrium with $I\equiv0$, whereas an EE $(S,I)$  is a steady state with $I(x)>0$ for some $x\in\Omega$.
The authors were particularly interested in the asymptotic behavior of the unique EE as $d_S$ approaches zero. Among other things,
their results imply that, if the spatial environment can be modified to
include low-risk sites and the movement of susceptible individuals can be restricted, then it may be possible to eliminate the infectious disease.

Although the existence and uniqueness of EE is proved in \cite{Allen} when $\mathcal R_0>1$, the global stability of it was left open. In some special cases,
the authors of \cite{PL09-NA} confirmed that it is indeed globally asymptotically stable. Further results concerning the asymptotic behavior of the EE of \eqref{AllenSIS}
were obtained by \cite{Peng-JDE09,Peng-Yi}. On the other hand, with $\beta$ and $\gamma$ being functions of spatiotemporal variables and
temporally periodic, the model  \eqref{AllenSIS} was treated by the second author and Zhao \cite{Peng-Zhao-2012},  and the theoretical findings of \cite{Peng-Zhao-2012} imply that the combination of spatial heterogeneity and temporal periodicity can enhance
the persistence of the disease. We refer interested readers to \cite{Allen1, CTZ, CLL, CL, DW,DHK, GR, GKLZ,HHL, KMP, Li-Peng-Wang,WZ} and the references therein
for related research work on \eqref{AllenSIS}.

The model \eqref{SIS-2} with $\chi=0$ was studied by Li et al.  \cite{Li-Peng-Wang}, where comprehensive qualitative analysis has been performed and the findings indicate  that a
varying total population can enhance persistence of infectious disease, and hence the disease becomes
more threatening and harder to control.

Biologically, the cross-diffusion introduced to the systems \eqref{SIS-1} and \eqref{SIS-2} represents a strategy that the susceptible implements to avoid infection by staying away from the infected (known as the repulsive chemotaxis phenomenon \cite{CLM-2008, Li-Pan-Zhao, Tao-2013, Tao-Wang}). The main purpose of this paper is to investigate the influence of such directed movement strategy  of the susceptible population on the persistence or extinction of infectious diseases
in the environment of spatial heterogeneity and random population movement via performing qualitative analysis on the systems \eqref{SIS-1} and \eqref{SIS-2}.
The cross-diffusion term  $\chi\nabla \cdot(S\nabla I)$ has been widely shown  to have a strong  effect in driving solutions of the underlying models to blow up in finite/infinite time, as can be seen in the extensively studied  Keller-Segel chemotaxis  related  systems \cite{BBTW15, JL92-TAMS, Win11, Win13}. Thus, the global solvability of systems \eqref{SIS-1} and \eqref{SIS-2} needs to be seriously treated  before we study their other dynamical properties. By a close inspection of the $I$-equation in \eqref{SIS-1} or \eqref{SIS-2}, we find that the essential  linearity  not only enables us to obtain the $L^\infty$-bound of $I$, but also that of $\nabla I$, while this information is usually unavailable in  most Keller-Segel models. With such a key observation, we are then able to establish the global existence and boundedness of classical solutions to \eqref{SIS-1} and  \eqref{SIS-2} for arbitrary  $\chi>0$ in any spatial dimensions; see Theorem \ref{glo-sol}. This result shows the cross-diffusion does not destroy the global solvability of the corresponding  system without cross-diffusion.

As in \cite{Allen}, for our systems \eqref{SIS-1} and \eqref{SIS-2}, we use the same definition of the basic reproduction number $\mathcal R_0$ since it determines the local stability of the unique DFE. Then we are also able to establish the threshold type dynamics in terms of  $\mathcal R_0$. More specifically, we show that the unique DFE is in fact globally stable if $\mathcal R_0<1$ (see Theorems \ref{glo-sta-1} and \ref{glo-sta-2}), which yield the extinction of infectious disease in the long run. While in the case of $\mathcal R_0>1$, a unique EE exists for system \eqref{SIS-1} whereas its uniqueness is unclear for system \eqref{SIS-2} since we are no longer able to reduce the equilibrium problem to a single equation due to the non-conservation of total population. In the special case that the transmission rate is proportional to the recovery rate throughout the habitat, it is proved that the unique homogeneous EE is globally stable when $\mathcal R_0>1$, provided that $\chi>0$ is suitably small; see Theorems \ref{sta-ee} and \ref{sta-ee-2}, which cover and extend \cite[Theorem 1.2]{PL09-NA} with $\chi=0$. Compared to the no cross-diffusion system \eqref{AllenSIS}, our results suggest that such directed movement strategy adopted by the susceptible with insignificant magnitude does not help to eliminate the infectious disease.

To study the effect of random motility of susceptible populations, we discuss the asymptotic behavior of the EE as $d_S\to 0$. For system \eqref{SIS-1} with constant total population, whenever $\mathcal R_0>1$ and the domain includes points where the transmission rate is smaller than the recovery rate, it is shown that the unique EE tends to a spatially inhomogeneous DFE as $d_S\to 0$. Furthermore, the density of the susceptible population of this limiting DFE, positive on {\it low-risk sites} (where the transmission rate is less than the recovery rate, i.e., where $\beta(x)<\gamma(x)$), must also be positive at some (but not all) {\it high-risk sites} (where the transmission rate is larger than the recovery rate, i.e., where $\beta(x)>\gamma(x)$). This result agrees with that of \cite{Allen} for model \eqref{AllenSIS} without directed diffusion. From the biological point of view, this in particular means that it is possible to eliminate the disease entirely in the habitat by restricting the random motility of susceptible individuals to be small.
In stark contrast, for model \eqref{SIS-2} with varying total population, although we are not able to fully determine the asymptotic profile of EE for small $d_S>0$,
Theorem \ref{asy-ee-prob2} below implies that the disease still exists on the whole habitat for any given $\chi>0$, and therefore the introduction of cross-diffusion for the susceptible can not help to eliminate the disease.
As a consequence, the theoretical finding in the current paper, in combination with the result of \cite{Li-Peng-Wang}, suggests that the restriction of the diffusion rate of the susceptible is no longer an appropriate strategy for the eradication of infectious disease modelled by \eqref{SIS-2} where the total population number can vary.

The plan of this paper is organized as follows. In Section 2, we discuss the global existence and boundedness of solutions to models \eqref{SIS-1} and \eqref{SIS-2} based on a semigroup type argument. Section 3 is devoted to the threshold dynamics where the global stability of DFE and EE (in a special case) is studied. In Section 4, by reducing the equilibrium problem of \eqref{SIS-1} to a single equation, we establish the existence and uniqueness of EE. Asymptotic profile of the EE for small $d_S$ is then discussed in Section 5. Finally, in Section 6, we briefly investigate system \eqref{SIS-2} and point out the main differences.

\section{Global Existence and Boundedness}

In this section, we shall establish the global existence and boundedness property of classical solutions to \eqref{SIS-1} and \eqref{SIS-2} via semigroup theory. For the sake of reference, we present some known smoothing $L^p$-$L^q$ type estimates on the Neumann heat semigroup $\left(e^{tk\Delta}\right)_{t\geq0}$ on a bounded and smooth domain $\Omega$. One can find them in \cite[Lemma 1.3]{Win10-JDE}, \cite[Lemma 2.1]{Cao15} or \cite[Lemma 2.1]{JX-Non16}.

\begin{lemma}
\label{semigroup}
For $k>0$, let $\left(e^{tk\Delta}\right)_{t\geq0}$ be the Neumann heat semigroup and $\lambda_1=:\lambda_1(k)>0$ be the first positive Neumann eigenvalue of $-k\Delta$ on $\Omega$. Then there exist some positive constants $c_i$ {\rm (}$i=1,2,3,4${\rm )} depending only on $k$ and $\Omega$ fulfilling
\begin{enumerate}
\item[{\rm (i)}] If $1\leq q\leq p\leq \infty$, then
\begin{equation}
\nonumber
\left\|e^{tk\Delta}f\right\|_{L^p(\Omega)}\leq c_1 \left(1+t^{-\frac{n}{2}\left(\frac{1}{q}-\frac{1}{p} \right)} \right)\left\|f\right\|_{L^q(\Omega)},\quad\forall t>0
\end{equation}
holds for all $f\in L^q(\Omega)$.

\item[{\rm (ii)}] If $1\leq q\leq p\leq \infty$, then
\begin{equation}
\nonumber
\left\|\nabla e^{tk\Delta}f\right\|_{L^p(\Omega)}\leq c_2 \left(1+t^{-\frac{1}{2}-\frac{n}{2}\left(\frac{1}{q}-\frac{1}{p} \right)} \right)e^{-\lambda_1t}\left\|f\right\|_{L^q(\Omega)},\quad\forall t>0
\end{equation}
holds for all $f\in L^q(\Omega)$.

\item[{\rm (iii)}] If $2\leq q\leq p<\infty$, then
\begin{equation}
\nonumber
\left\|\nabla e^{tk\Delta}f\right\|_{L^p(\Omega)}\leq c_3 \left(1+t^{-\frac{n}{2}\left(\frac{1}{q}-\frac{1}{p} \right)} \right)e^{-\lambda_1t}\left\|\nabla f\right\|_{L^q(\Omega)},\quad\forall t>0
\end{equation}
holds for all $f\in W^{1,q}(\Omega)$.

\item[{\rm (iv)}] If $1<q\leq p\leq\infty$, then
\begin{equation}
\nonumber
\left\|e^{tk\Delta}\nabla \cdot f\right\|_{L^p(\Omega)}\leq c_4 \left(1+t^{-\frac{1}{2}-\frac{n}{2}\left(\frac{1}{q}-\frac{1}{p} \right)} \right)e^{-\lambda_1t}\left\| f\right\|_{L^q(\Omega)},\quad\forall t>0
\end{equation}
holds for all $f\in \left(L^q(\Omega)\right)^n$.
\end{enumerate}
\end{lemma}

For notational convenience, throughout the paper, we shall denote
$$m^*=\max_{x\in\overline\Omega}m(x)\quad \mbox{and}\quad m_*=\min_{x\in\overline\Omega}m(x)$$
with $m\in \{\beta, \gamma\}$.

Using Lemma \ref{semigroup} and Banach's contraction mapping theorem, one can establish the local solvability of systems \eqref{SIS-1} and \eqref{SIS-2} . For details of the similar reasoning, we refer to \cite[Theorem 3.1]{HW05} and \cite[Lemma 1.1]{Win10-CPDE}; see also \cite[Lemma 3.1]{BBTW15}.

\begin{lemma}
\label{loc-sol}
Assume that the initial data fulfills \eqref{initial}. Then there exists $T_{\rm max}\in (0,\infty]$ and a uniquely determined pair of nonnegative functions
\begin{eqnarray}
&&\disp S\in C(\overline\Omega \times[0,T_{\rm max}))\cap C^{2,1}(\overline\Omega\times(0,T_{\rm max}))  \nonumber \\
&&\disp I\in C(\overline\Omega \times[0,T_{\rm max}))\cap C^{2,1}(\overline\Omega\times(0,T_{\rm max}))\cap L^\infty_{\rm loc}([0,T_{\rm max});W^{1,p}(\Omega))
\nonumber
\end{eqnarray}
with any $p>1$ such that $(S, I)$ solves \eqref{SIS-1} classically in $\Omega\times(0,T_{\rm max})$. Furthermore, if $T_{\rm max}<\infty$, then, for any $p>1$,
\begin{equation}
\|S(\cdot,t)\|_{L^\infty(\Omega)}+\|I(\cdot,t)\|_{W^{1,p}(\Omega)}\to\infty \mbox{ as } t\nearrow T_{\rm max}.
\label{extend}
\end{equation}
The same local-in-time well-posedness  holds true for model \eqref{SIS-2}.  For model \eqref{SIS-1}, the conservation law \eqref{L1-bdd} holds in $(0,T_{\rm max})$; for model \eqref{SIS-2}, the following uniform $L^1$-estimate for $S+I$ holds in $(0,T_{\rm max})$:
\be
\begin{split}
&\int_\Omega \left(S(x,t)+(1+\frac{1}{2\beta^*})I(x,t)\right)dx\\
&\leq \int_\Omega \left(S_0(x)+(1+\frac{1}{2\beta^*})I_0(x)\right)dx+ \frac{2\int_\Omega \Lambda(x)dx}{\min\{1, \frac{2\gamma_*}{1+2\beta^*}\}}=:\hat{N},\quad \forall t>0.
\end{split}
\label{L1-bdd2}
\ee
\end{lemma}
\begin{proof} As noted above, the statements concerning the local-in-time existence of  classical solutions to the initial-boundary value problems \eqref{SIS-1} and   \eqref{SIS-2} and the criterion \eqref{extend} are well-studied. The nonnegativity (positivity) of $(S, I)$ follows simply from the maximum principle.  Due to  no flux boundary conditions, upon integration of the $S$- and $I$-equation in \eqref{SIS-1}, the conservation law \eqref{L1-bdd} follows trivially.    For the uniform  $L^1$-bound  in  \eqref{L1-bdd2}, by straightforward computations, we deduce from the IBVP \eqref{SIS-2} that
\be\nonumber
\begin{split}
&\frac{d}{dt}\int_\Omega \left(S(x,t)+(1+\frac{1}{2\beta^*})I(x,t)\right)dx\\
&=\int_\Omega \Lambda(x)dx-\int_\Omega S(x,t)dx+\frac{1}{2\beta^*}\int_\Omega \frac{\beta(x)S(x,t) I(x,t)}{S(x,t)+I(x,t)}dx-\frac{1}{2\beta^*}\int_\Omega  \gamma(x) I(x,t)dx\\
&\leq \int_\Omega \Lambda(x)dx-\frac{1}{2}\int_\Omega S(x,t)dx-\frac{\gamma_*}{2\beta^*}\int_\Omega I(x,t)dx\\
&\leq \int_\Omega \Lambda(x)dx-\frac{1}{2}\min\left\{1, \frac{2\gamma_*}{1+2\beta^*}\right\}\int_\Omega \left(S(x,t)+(1+\frac{1}{2\beta^*})I(x,t)\right)dx.
\end{split}
\ee
Solving this standard Gronwall differential inequality, we arrive at \eqref{L1-bdd2}.
\end{proof}
 Our main result on global existence and uniform-in-time boundedness for \eqref{SIS-1} and \eqref{SIS-2} reads precisely as follows.
\begin{theorem}
\label{glo-sol}
Assume that the initial data fulfills \eqref{initial}. Then each of the cross-diffusive SIS models \eqref{SIS-1} and \eqref{SIS-2} possesses a uniquely determined global classical  solution $(S, I)$ for which both $S$ and $I$ are positive  and bounded in $\overline\Omega\times (0,\infty)$. That is, there exists some $M>0$ depending on initial data and the model parameters such that
\be\label{bdd0}
\|S(\cdot,t)\|_{L^\infty(\Omega)}+\|I(\cdot,t)\|_{W^{1,\infty}(\Omega)}\leq M,\quad \forall t>0.
\ee
Moreover, there exists some $M'>0$ independent of initial data fulfilling
\be
\|S(\cdot,t)\|_{L^\infty(\Omega)}+\|I(\cdot,t)\|_{L^\infty(\Omega)}\leq M',\quad \forall t>T
\label{bdd}
\ee
for some large $T>0$. Furthermore, the $L^\infty$-bound of $I$ is uniform in $\chi$, i.e.,
\be\label{I-ub}
\|I(\cdot,t)\|_{L^\infty(\Omega)}\leq M_I(n, \Omega, \gamma_*,\beta^*) \left(1+\frac{1}{d_I}\right)^n \max\left\{\|I_0\|_{L^\infty(\Omega)},\tilde{N}\right\},\quad \forall t>0.
\ee
Here, $\tilde{N}=N$ for model \eqref{SIS-1} and $\tilde{N}=\hat{N}$ with $\hat{N}$ defined by \eqref{L1-bdd2} for model \eqref{SIS-2}.
\end{theorem}

\begin{proof}
Thanks to Lemma \ref{loc-sol}, we shall first show $T_{\rm max}=\infty$ and then the global boundedness of $(S,I)$. To this end,  we start with the $I$-associated  problem
\be
\left\{ \begin{array}{ll}
\disp I_t=d_I\Delta I+B(x,t)I,&x\in\Omega,t>0,\\
\disp \frac{\partial I}{\partial\nu}=0,&x\in\partial\Omega,t>0,\\
\disp I(x,0)=I_0(x),&x\in\Omega,
\end{array}\right.
\label{I-linear}
\ee
where $$B(x,t)=\beta(x)\frac{S(x,t)}{S(x,t)+I(x,t)}-\gamma(x).$$
It is clear that $B$ is uniformly bounded by $(\beta^*+\gamma^*)$ and is locally Lipschitz on $\Omega\times (0,T_{\rm max})$. Furthermore, $\|I(\cdot,t)\|_{L^1(\Omega)}\leq \tilde{N}$ for $t\in (0,T_{\rm max})$ due to the validity of \eqref{L1-bdd} and \eqref{L1-bdd2}  in $(0,T_{\rm max})$. Thus,  \cite[Theorem 3.1  on \lq\lq$L^1$-boundedness implies $L^\infty$-boundedness"]{Al79} (see also \cite[Lemma 3.1]{Peng-Zhao-2012}) applied to \eqref{I-linear} yields the existence of a positive constant $C_1$ such that
\be
\|I(\cdot,t)\|_{L^\infty(\Omega)}\leq C_1(d_I,\Omega, n,\beta, \gamma),\quad \forall t\in(0,T_{\rm max}).
\label{I-bdd}
\ee
Thanks to the $L^1$-bound in \eqref{L1-bdd} and \eqref{L1-bdd2}, the bound for $I$ in \eqref{I-ub} indeed could be obtained via standard  Moser iteration applied to \eqref{I-linear}.

Next, according to the variation-of-constants formula, we have
$$I(\cdot,t)=e^{t d_I \Delta}I_0+\int_0^t e^{(t-\tau)d_I\Delta} \left(\frac{\beta SI}{S+I}-\gamma I\right)(\cdot,\tau)d\tau, \quad \forall t\in(0,T_{\rm max}),$$
from which it follows that
\be
\label{grad-I}
\begin{split}
\left\|\nabla I(\cdot, t)\right\|_{L^\infty(\Omega)}&\leq \left\|\nabla e^{t d_I \Delta}I_0 \right\|_{L^\infty(\Omega)}\\
&\quad+\int_0^t \left\| \nabla e^{(t-\tau)d_I\Delta} \left(\frac{\beta SI}{S+I}-\gamma I\right)(\cdot,\tau)\right\|_{L^\infty(\Omega)}d\tau.
\end{split}
\ee
For  $p\in [2,\infty)$, Lemma \ref{semigroup} (iii) entails, for all $t>0$,
$$
\left\|\nabla e^{td_I\Delta}I_0\right\|_{L^p(\Omega)}\leq 2c_3e^{-\lambda_1(d_I)t}\left\|\nabla I_0\right\|_{L^p(\Omega)}\leq 2c_3e^{-\lambda_1(d_I)t}\max\{|\Omega|^\frac{1}{2}, 1\}\left\|\nabla I_0\right\|_{L^\infty(\Omega)}. 
$$
This immediately implies the existence of a constant $C_2>0$ such that
\be\label{grad-I-term1}
\left\| \nabla e^{t d_I \Delta}I_0\right\|_{L^\infty(\Omega)}\leq C_2 \left\| I_0\right\|_{W^{1,\infty}(\Omega)},\quad \forall t>0.
\ee
On the other hand, it follows from Lemma \ref{semigroup} (ii) that
\be\nonumber
\begin{split}
&\left\| \nabla e^{(t-\tau)d_I\Delta} \left(\frac{\beta SI}{S+I}-\gamma I\right)(\cdot,\tau)\right\|_{L^\infty(\Omega)}\\
&\leq c_2\left(1+(t-\tau)^{-\frac{1}{2}}\right)e^{-\lambda_1(d_I) (t-\tau)}  (\beta^*+\gamma^*)\left\| I(\cdot,\tau)\right\|_{L^\infty(\Omega)},\quad \forall t\in (\tau,T_{\rm max}).
\end{split}
\ee
This along with \eqref{grad-I}, \eqref{grad-I-term1} and the $L^\infty$-boundedness of $I$ yields that
\begin{eqnarray}
&&\left\|\nabla I(\cdot, t)\right\|_{L^\infty(\Omega)}\leq C_2 \left\| I_0\right\|_{W^{1,\infty}(\Omega)}+C_3\int_0^t \left(1+(t-\tau)^{-\frac{1}{2}}\right)e^{-\lambda_1(d_I) (t-\tau)}d\tau \nonumber \\
&&\quad \quad \quad \quad \quad \quad \quad  \leq C_4,\quad \forall t\in (0,T_{\rm max}).\label{grad-I-bdd}
\end{eqnarray}
Now, we are ready to derive the $L^\infty$-bound of $S$. For definiteness, we  will first work on  \eqref{SIS-1} and just  give a quick  remark for \eqref{SIS-2} in the end. Fix any $0<T<T_{\rm max}$ and $p>n$. We rewrite the $S$-equation in \eqref{SIS-1}  as
$$S_t-d_S\Delta S+S=\chi \nabla\cdot (S\nabla I)+S-\beta\frac{ S}{S+I}I+\gamma I,$$
which gives, upon an application of the variation-of-constants formula,
\be
\nonumber
\begin{split}
S(\cdot,t)&=e^{t(d_S\Delta -1)}S_0+\chi \int_0^t e^{(t-\tau)(d_S\Delta-1)}\nabla \cdot \left(S(\cdot,\tau)\nabla I(\cdot,\tau) \right)d\tau\\
&\quad \ +\int_0^t e^{(t-\tau)(d_S\Delta-1)}\left(S-\beta\frac{S}{S+I}I+\gamma I \right)(\cdot,\tau)d\tau,\quad \forall t\in (0,T).
\end{split}
\ee
Taking supremum on both sides, we obtain
\begin{align}
 \left\| S(\cdot,t)\right\|_{L^\infty(\Omega)} &\leq \left\|e^{t(d_S\Delta -1)}S_0\right\|_{L^\infty(\Omega)} \nonumber \\
& \quad \ +\chi \int_0^t \left\| e^{(t-\tau)(d_S\Delta-1)}\nabla \cdot \left(S(\cdot,\tau)\nabla I(\cdot,\tau) \right)\right\|_{L^\infty(\Omega)} d\tau \nonumber \\
& \quad \ +\int_0^t \left\| e^{(t-\tau)(d_S\Delta-1)}\left(S-\beta\frac{S}{S+I}I+\gamma I \right)(\cdot,\tau)\right\|_{L^\infty(\Omega)} d\tau \nonumber \\
& =:{\rm I}+{\rm II}+{\rm III},\quad \forall t\in (0,T).\label{S-est}
\end{align}

We first deduce from the maximum principle that
\be
\label{term-I}
{\rm I} = e^{-t}\left\| e^{td_S\Delta}S_0\right\|_{L^\infty(\Omega)}\leq e^{-t}\left\| S_0\right\|_{L^\infty(\Omega)},\quad \forall t\in(0,T).
\ee
For convenience, we define
$$R=R(T):=\sup_{t\in (0,T)}\left\| S(\cdot, t)\right\|_{L^\infty(\Omega)}.$$
Using Lemma \ref{semigroup} (iv) and \eqref{grad-I-bdd}, we are led to
\begin{align}
{\rm II} &\leq \chi \int_0^t e^{-(t-\tau)}\left\| e^{(t-\tau)d_S\Delta}\nabla\cdot \left(S(\cdot,\tau)\nabla I(\cdot,\tau)\right)\right\|_{L^\infty(\Omega)}d\tau \nonumber \\
&  \leq C_5\int_0^t e^{-(t-\tau)}\left(1+(t-\tau)^{-\frac{1}{2}-\frac{n}{2p}}\right)\left\| S(\cdot,\tau)\nabla I(\cdot,\tau)\right\|_{L^p(\Omega)}d\tau \nonumber \\
&  \leq C_6\int_0^t e^{-(t-\tau)} \left(1+(t-\tau)^{-\frac{1}{2}-\frac{n}{2p}}\right)\left\| S(\cdot,\tau)\right\|_{L^p(\Omega)}d\tau \nonumber \\
& \leq C_6 \int_0^t e^{-(t-\tau)} \left(1+(t-\tau)^{-\frac{1}{2}-\frac{n}{2p}}\right)\left\| S(\cdot,\tau)\right\|_{L^\infty(\Omega)}^{1-\frac{1}{p}} \left\| S(\cdot,\tau)\right\|_{L^1(\Omega)}^{\frac{1}{p}} d\tau \nonumber \\
& \leq C_6 R^{1-\frac{1}{p}} N^{\frac{1}{p}} \int_0^t e^{-(t-\tau)} \left(1+(t-\tau)^{-\frac{1}{2}-\frac{n}{2p}}\right) d\tau \nonumber \\
& \leq C_6 C_7 R^{1-\frac{1}{p}} N^{\frac{1}{p}},
\quad \forall t\in (0,T),\label{term-II}
\end{align}
where the fact that $p>n$ was used to guarantee
\begin{align*}
C_7&=\int_0^t e^{-(t-\tau)} \left(1+(t-\tau)^{-\frac{1}{2}-\frac{n}{2p}}\right) d\tau\\
&\leq 1+\int_0^\infty e^{-\sigma}\sigma ^{-\frac{1}{2}-\frac{n}{2p}}d\sigma=1+\Gamma\left(\frac{1}{2}-\frac{n}{2p}\right)<\infty.
\end{align*}
 To estimate  ${\rm III}$, using \eqref{L1-bdd} and Lemma \ref{semigroup} (i), we deduce
\begin{align}
{\rm III} &\leq \int_0^t e^{-(t-\tau)} \left\|e^{(t-\tau)d_S\Delta} S(\cdot,\tau) \right\|_{L^\infty(\Omega)}d\tau \nonumber \\
& \quad \ +\int_0^t e^{-(t-\tau)} \left\| e^{(t-\tau)d_S\Delta} \left(-\beta\frac{S}{S+I}I+\gamma I \right)(\cdot,\tau) \right\|_{L^\infty(\Omega)}d\tau \nonumber \\
& \leq c_1 \int_0^t e^{-(t-\tau)} \left(1+(t-\tau)^{-\frac{n}{2p}}\right)\left\| S(\cdot,\tau)\right\|_{L^p(\Omega)}d\tau \nonumber \\
& \quad \ +\int_0^t e^{-(t-\tau)}  (\beta^*+\gamma^*)\left\| I(\cdot,\tau)\right\|_{L^\infty(\Omega)}d\tau \nonumber \\
&  \leq c_1 \int_0^t e^{-(t-\tau)} \left(1+(t-\tau)^{-\frac{n}{2p}}\right)\left\| S(\cdot,\tau)\right\|_{L^\infty(\Omega)}^{1-\frac{1}{p}} \left\| S(\cdot,\tau)\right\|_{L^1(\Omega)}^{\frac{1}{p}} d\tau+C_8 \nonumber \\
& \leq c_1 C_{9} R^{1-\frac{1}{p}} N^{\frac{1}{p}}+C_8 \label{term-III}
\end{align}
with
$$C_{9}=\int_0^\infty e^{-\sigma}\left(1+\sigma^{-\frac{n}{2p}}\right)d\sigma<\infty.$$
Substituting \eqref{term-I}, \eqref{term-II} and \eqref{term-III} into  \eqref{S-est}, we conclude that
$$R\leq C_{10}R^{1-\frac{1}{p}}+C_{11},$$
with $C_{10}=(C_6C_7+c_1 C_{9})N^{\frac{1}{p}}$ and $C_{11}=\|S_0\|_{L^\infty(\Omega)}+C_8$. Upon an elementary argument, we infer from the above inequality that
\be
\label{S-bound}
R\leq \max\left\{ \left(\frac{C_{11}}{C_{10}}\right)^{\frac{p}{p-1}}, \ \ (2C_{10})^p\right\}.
\ee
Since $T\in (0,T_{\rm max})$ is arbitrary, we conclude  there must  exist  $C>0$ such that
$$\left\| S(\cdot,t)\right\|_{L^\infty(\Omega)}\leq C,\quad\forall t\in(0,T_{\rm max}).$$
This in conjunction with  \eqref{I-bdd}, \eqref{grad-I-bdd} and \eqref{extend} indicates that $T_{\rm max}=\infty$.

Now, we proceed to find an upper bound of $\|S(\cdot,t)+I(\cdot,t)\|_{L^\infty(\Omega)}$ that is independent of initial data for large $t$. In fact,  a use of \cite[Lemma 3.1]{Peng-Zhao-2012} to problem \eqref{I-linear}  provides some constant $C_{12}>0$ independent of initial data and $T_1>0$ such  that
$$\|I(\cdot,t)\|_{L^\infty(\Omega)}\leq C_{12},\quad t\geq T_1.$$
Next, we represent $I$ in the following way:
$$I(\cdot,t)=e^{(t-T_1)(d_I\Delta-1)}I(\cdot,T_1)+\int_{T_1}^t e^{(t-\tau)(d_I\Delta-1)} \left(I +\beta\frac{SI}{S+I}-\gamma I\right)(\cdot,\tau)d\tau,\quad\forall t>T_1.$$
Then using a parallel argument to the one used to derive  \eqref{grad-I-bdd}, it is easily seen  there exists $T_2>T_1$, such that
$$\|\nabla I(\cdot,t)\|_{L^\infty(\Omega)}\leq C_{13},\quad\forall t\geq T_2,$$
for some constant $C_{13}>0$ independent of the initial data. Next representing $S$ as
\begin{align*}
S(\cdot,t)&=e^{(t-T_2)(d_S\Delta -1)}S(\cdot,T_2)+\chi \int_{T_2}^t e^{(t-\tau)(d_S\Delta-1)}\nabla \cdot \left(S(\cdot,\tau)\nabla I(\cdot,\tau) \right)d\tau\\
&\quad +\int_{T_2}^t e^{(t-\tau)(d_S\Delta-1)}\left(S-\beta\frac{S}{S+I}I+\gamma I \right)(\cdot,\tau)d\tau,\quad \forall t>T_2,
\end{align*}
and proceeding in the same fashion as we did to obtain  \eqref{S-bound}, we can conclude that actually $\|S(\cdot,t)\|_{L^\infty(\Omega)}$ can be bounded by a positive constant independent of the initial data for $t$ sufficiently large. This proves the theorem for model \eqref{SIS-1}.

 Armed with the uniform $W^{1,\infty}$-bound for $I$ as obtained in \eqref{I-bdd} and \eqref{grad-I-bdd} and the uniform  $L^1$-bound for $S+I$ in \eqref{L1-bdd2},   with minor  modifications,  we can repeat the proceeding arguments for \eqref{SIS-1} to derive the assertions of the theorem for \eqref{SIS-2}.
\end{proof}

\section{Threshold Dynamics in terms of $\mathcal{R}_0$}

In this section, we aim to study the global asymptotic stability of nonnegative steady states of \eqref{SIS-1} with the  constraint \eqref{L1-bdd}.
It is straightforward to see that the DFE exists uniquely and is given by
$$(\tilde S,0)= \left(\frac{N}{|\Omega|},0\right).$$

For our model \eqref{SIS-1}, as in \cite{Allen}, we define the basic reproduction number by
\be\label{def-R0}\mathcal{R}_0=\sup_{0\neq \varphi\in H^1(\Omega)} \frac{\int_\Omega \beta \varphi^2 dx}{\int_\Omega\left(d_I|\nabla \varphi|^2+\gamma \varphi^2\right)dx }.\ee
Indeed, one can follow the idea of next generation operators in \cite{Peng-Zhao-2012} to introduce the basic reproduction number,
which coincides with $\mathcal{R}_0$. Observe that $\mathcal{R}_0$ is independent of the diffusion rate $d_S>0$.

\begin{proposition}\label{loc-sta}
The DFE is stable if $\mathcal{R}_0<1$, and it is unstable if $\mathcal{R}_0>1.$
\end{proposition}

The proof of Proposition \ref{loc-sta} is the same as that of \cite[Lemma 2.4]{Allen} and so the details are omitted here.
In addition, the following qualitative properties of $\mathcal{R}_0$ were also established in \cite{Allen}.

\begin{proposition}\label{property} The following assertions hold.

\begin{enumerate}

 \item[\rm{(a)}] $\mathcal{R}_0$ is a monotone decreasing function of $d_I$
with $\mathcal{R}_0\rightarrow\max\{\beta(x)/{\gamma(x)}:\
x\in\overline\Omega\}$ as $d_I\rightarrow0$ and
$\mathcal{R}_0\rightarrow\int_\Omega\beta/{\int_\Omega\gamma}$ as
$d_I\rightarrow\infty$;

 \item[\rm{(b)}] if $\int_\Omega \beta(x)\mbox{d}x<\int_\Omega \gamma(x)\mbox{d}x$, then there exists a
threshold value $d_I^*\in(0,\infty)$ such that $\mathcal{R}_0>1$ for
$d_I<d_I^*$ and $\mathcal{R}_0<1$ for $d_I>d_I^*$;

 \item[\rm{(c)}] if $\int_\Omega \beta(x)\mbox{d}x\geq\int_\Omega \gamma(x)\mbox{d}x$, then $\mathcal{R}_0>1$ for all $d_I$.

\end{enumerate}
\end{proposition}

Furthermore, it was shown in \cite{Allen} that $1-\mathcal R_0$ has the same sign as $\lambda^*$ with $\lambda^*$ being the principal eigenvalue of the following eigenvalue problem.
\begin{equation}
\left\{ \begin{array}{ll}
\displaystyle -d_I\Delta \phi+(\gamma-\beta)\phi=\lambda \phi,&x\in\Omega,\\
\displaystyle \frac{\partial \phi}{\partial\nu}=0,&x\in\partial\Omega.
\end{array}\right.
\label{eigen-prob}
\end{equation}

Notice that the $I$-equation verifies
\be\label{I-com-ee}
I_t=d_I\Delta I+\beta\frac{SI}{S+I}-\gamma I\leq d_I\Delta I+(\beta-\gamma)I,\quad\quad x\in\Omega,t>0.
\ee
Based on \eqref{eigen-prob} and \eqref{I-com-ee}, we first have a simple observation, when $\mathcal R_0<1$, that  $I(x,t)$ decays to $0$  exponentially as $t\to\infty$. In fact, in such case,  the principal eigenvalue $\lambda^*$ corresponding to \eqref{eigen-prob} is positive. Let $\phi^*>0$ be the principal eigenfunction and let $M$  be a positive constant such that $I_0(x)\leq M \phi^*(x)$ for $x\in\Omega$.  Then a direct application from \eqref{eigen-prob} and \eqref{I-com-ee} shows that $Me^{-\lambda^* t}\phi^*(x)$ is a super-solution to the $I$-equation and hence  the comparison principle yields that
\begin{equation}
I(x,t)\leq Me^{-\lambda^* t}\phi^*(x),\quad\quad x\in\Omega,t>0.
\label{I-decay}
\end{equation}
In the sequel, we shall employ this important information to derive the global asymptotic stability of DFE under the assumption that $\mathcal R_0<1$, and this is achieved through a chain of simple lemmas.

For convenience, we set $w(x,t)=S(x,t)+I(x,t)$ for $x\in\overline{\Omega}$ and $t\geq 0$. Then it can be readily checked from the IBVP \eqref{SIS-1} that $w$ satisfies
\begin{equation}
\left\{ \begin{array}{lll}
\displaystyle w_t=d_S\Delta w+\nabla \cdot \left(\left(d_I-d_S\right)\nabla I+\chi S\nabla I \right),&x\in\Omega, t>0,\\
\displaystyle \frac{\partial w}{\partial \nu}=0,&x\in\partial\Omega, t>0,\\
\displaystyle w(x,0)=S_0(x)+I_0(x),&x\in\Omega.
\end{array}\right.
\label{w}
\end{equation}
 The following result serves as a key starting point towards our stabilization analysis for the DFE $(\tilde S,0)$.
\begin{lemma}
\label{w-decay}
 The ansatz  $w$  satisfies the property  that
\begin{equation}
\int_0^\infty \int_\Omega |\nabla w|^2dxdt<\infty.
\label{nabla-w}
\end{equation}
\end{lemma}

\begin{proof}
Multiplying the $I$-equation by $I$ and integrating over $\Omega$, we are led to
\begin{eqnarray}
\frac{1}{2}\frac{d}{dt}\int_\Omega I^2dx=-d_I\int_\Omega |\nabla I|^2dx+\int_\Omega \left(\beta\frac{SI}{S+I}-\gamma I\right)I dx.
\nonumber
\end{eqnarray}
Thanks to \eqref{I-decay}, we have
\begin{align}
 d_I\int_0^T\int_\Omega |\nabla I|^2dxdt &=\frac{1}{2}\int_\Omega I_0^2dx-\frac{1}{2}\int_\Omega I^2(\cdot,T)dx+\int_0^T\int_\Omega \left(\beta \frac{S}{S+I}-\gamma \right)I^2dxdt\nonumber\\
&  \leq  \frac{1}{2}\int_\Omega I_0^2dx+\beta^* \int_0^T\int_\Omega \left(Me^{-\lambda^*t}\phi^*(x) \right)^2dxdt\nonumber\\
&  \leq C_1,\quad \forall T>0,
\label{nabla-I}
\end{align}
for some positive constant $C_1$ independent of $T>0$. Similarly, from the $S$-equation and the fact that $S$ and $I$ are bounded, we deduce
\begin{align}
 \frac{1}{2}\frac{d}{dt}\int_\Omega S^2dx &=-d_S\int_\Omega |\nabla S|^2dx-\chi \int_\Omega S\nabla I\cdot\nabla Sdx-\int_\Omega\beta\frac{S^2I}{S+I}dx+\int_\Omega \gamma SIdx\nonumber\\
&  \leq -d_S\int_\Omega |\nabla S|^2dx+\frac{d_S}{2}\int_\Omega |\nabla S|^2dx+\frac{\chi^2}{2d_S}\int_\Omega S^2|\nabla I|^2dx\nonumber\\
&\ \quad +\gamma^*\|S(\cdot,t)\|_{L^\infty(\Omega)}\int_\Omega Idx\nonumber\\
&  \leq -\frac{d_S}{2}\int_\Omega |\nabla S|^2dx+\frac{\chi^2}{2d_S}\|S(\cdot,t)\|_{L^\infty(\Omega)}^2\int_\Omega |\nabla I|^2dx\nonumber\\
&\ \quad +\gamma^*\|S(\cdot,t)\|_{L^\infty(\Omega)}\int_\Omega Idx.\nonumber
\end{align}
As a result,
\begin{align}
 d_S\int_0^T\int_\Omega |\nabla S|^2dxdt & \leq \int_\Omega S_0^2dx+C_2\int_0^T\int_\Omega |\nabla I|^2dxdt+C_2\int_0^T\int_\Omega Idxdt\nonumber\\
& \leq C_3,\quad \forall T>0,
\label{nabla-S}
\end{align}
due  to \eqref{nabla-I} and \eqref{I-decay}. Clearly, by the definition of $w$, the desired estimate \eqref{nabla-w} follows from \eqref{nabla-I} and \eqref{nabla-S}.
\end{proof}

With the aid of the decaying property  \eqref{nabla-w}, we next show  that $w_t$ decays at least in the dual space of $H^1(\Omega)$ in the large time limit.
\begin{lemma}
\label{wt-decay}
The solution $w$ of \eqref{w} satisfies
$$\int_0^\infty \left\|w_t(\cdot,t) \right\|_{\left(H^1(\Omega)\right)^*}^2 dt<\infty.$$
\end{lemma}

\begin{proof}
For any test function $\varphi\in H^1(\Omega)$, we obtain from \eqref{w} that
\be
\nonumber
\begin{split}
\int_\Omega w_t\varphi dx&=-d_S \int_\Omega \nabla w\cdot\nabla \varphi dx-\int_\Omega \left[(d_I-d_S)\nabla I+\chi S\nabla I\right]\cdot \nabla \varphi dx\\
&\leq d_S \|\nabla w(\cdot,t)\|_{L^2(\Omega)}  \|\nabla \varphi \|_{L^2(\Omega)} \\
&\quad + \left[(d_I+d_S) \|\nabla I(\cdot,t)\|_{L^2(\Omega)}+\chi \|S(\cdot,t)\|_{L^\infty(\Omega)} \|\nabla I(\cdot,t)\|_{L^2(\Omega)} \right] \|\nabla \varphi \|_{L^2(\Omega)}\\[0.25cm]
&\leq \left[d_S \|\nabla w(\cdot,t)\|_{L^2(\Omega)} +\left(d_I+d_S+ \chi \|S(\cdot,t)\|_{L^\infty(\Omega)}\right) \|\nabla I(\cdot,t)\|_{L^2(\Omega)} \right]\| \varphi \|_{H^1(\Omega)}.
\end{split}
\ee
This shows precisely that
\be
\left\| w_t(\cdot,t)\right\|_{\left(H^1(\Omega)\right)^*}\leq d_S \|\nabla w(\cdot,t)\|_{L^2(\Omega)}+\left(d_I+d_S+ \chi \|S(\cdot,t)\|_{L^\infty(\Omega)}\right) \|\nabla I(\cdot,t)\|_{L^2(\Omega)}.
\nonumber
\ee
As a consequence of Theorem \ref{glo-sol}, Lemma \ref{w-decay} and \eqref{nabla-I}, we deduce
$$
\int_0^\infty \left\|w_t(\cdot,t) \right\|_{\left(H^1(\Omega)\right)^*}^2 dt\leq C \int_0^\infty \left\|\nabla w(\cdot,t) \right\|_{L^2(\Omega)}^2 dt + C \int_0^\infty \left\|\nabla I(\cdot,t) \right\|_{L^2(\Omega)}^2 dt<\infty.
$$
This completes the proof of the lemma.
\end{proof}

With the help of Lemmas \ref{w-decay} and \ref{wt-decay}, using a somewhat standard argument as in \cite{TW12-JDE} and \cite{Win14}, we now can establish the global stability of DFE.

\begin{theorem}
\label{glo-sta-1}
If the basic reproduction number $\mathcal R_0<1$, then the unique global-in-time  bounded and classical solution $(S, I)$  of \eqref{SIS-1} satisfies $(S,I)\to \left(\frac{N}{|\Omega|},0\right)$ in $C(\overline\Omega)$ as $t\to\infty.$
\end{theorem}

\begin{proof}
We have already shown in \eqref{I-decay} that $I(x,t)\to 0$ uniformly for $x\in\overline\Omega$ as $t\to\infty$. Recall that $w=S+I$, hence, it suffices to prove
\be
\left\| w(\cdot,t)-\frac{N}{|\Omega|}\right\|_{C(\overline\Omega)}\to 0,\ \ \mbox{ as }t\to\infty.
\label{glo-sta1}
\ee

We shall show \eqref{glo-sta1} by a contradiction argument. Suppose otherwise, then there exists a sequence $\{t_k\}_{k\in\mathbb N}$ with $t_k\to \infty$ as $k\to \infty$ such that
\be
\inf_{k\in \mathbb N} \left\| w(\cdot,t_k)-\frac{N}{|\Omega|}\right\|_{C(\overline\Omega)}>0.
\label{glo-sta2}
\ee
Since $\|S(\cdot,t)+I(\cdot,t)\|_{L^\infty(\Omega)}$ is bounded for $t>0$ by Theorem \ref{glo-sol},  the standard parabolic regularity for  bounded solutions of parabolic  equations   (c.f.  \cite{La, PV93-JDE} or \cite[Theorem A2]{BDG81}) applied first to the $I$-equation in \eqref{SIS-1} tells us   that  $I(\cdot,t)$ is bounded in $C^2(\overline{\Omega})$ and then  the H\"{o}lder regularity  \cite[Theorem A1]{BDG81} applied to \eqref{w} (see more detailed reasonings leading to \eqref{reg-SI} below) gives us that  $\|S(\cdot,t)+I(\cdot,t)\|_{C^\alpha(\overline\Omega)}=\|w(\cdot,t)\|_{C^\alpha(\overline\Omega)}$ is bounded for all $t\geq 2$ for some $\alpha\in (0,1)$. An application of the Arzel\`{a}-Ascoli Theorem yields that $\{w(\cdot,t)\}_{t\geq 2}$ is relatively compact in $C(\overline\Omega)$. Thus, we can extract a subsequence, still denoted by $\{t_k\}_{k\in\mathbb N}$, such that
\be
w(\cdot,t_k)\to w_\infty\quad \mbox{in } C(\overline\Omega),\ \ \mbox{ as } k\to\infty
\label{glo-sta3}
\ee
for some $0\leq w_\infty\in C(\overline\Omega)$. First, the weak stabilization of $w_t$ in  Lemma \ref{wt-decay} entails
\begin{align*}
\int_{t_k}^{t_k+1}\left\| w(\cdot,t)-w(\cdot,t_k)\right\|^2_{\left(H^1(\Omega)\right)^*}dt & = \int_{t_k}^{t_k+1}\left\|\int_{t_k}^t w_t(\cdot,s)ds \right\|^2_{\left(H^1(\Omega)\right)^*}dt \\
& \leq \int_{t_k}^{t_k+1} \left(\int_{t_k}^t \|w_t(\cdot,s)\|^2_{\left(H^1(\Omega)\right)^*}ds \right) \cdot (t-t_k) dt  \\
& \leq \int_{t_k}^\infty  \|w_t(\cdot,s)\|^2_{\left(H^1(\Omega)\right)^*}ds\to 0,\ \mbox{ as } k\to\infty.
\end{align*}
 This along with \eqref{glo-sta3} and the continuous  embedding $L^\infty \hookrightarrow (H^1(\Omega))^*$ implies that
\be
\int_{t_k}^{t_k+1}\left\| w(\cdot,t)-w_\infty\right\|^2_{\left(H^1(\Omega)\right)^*}dt\to 0,\ \ \mbox{ as }k\to \infty.
\label{glo-sta4}
\ee
On the other hand, since $L^2(\Omega)\hookrightarrow \left(H^1(\Omega)\right)^*$ and $\int_\Omega w(x,t)dx=N$, we infer from  Lemma \ref{w-decay} and the Poincar\'e inequality  that
$$\int_0^\infty \left\| w(\cdot,t)-\frac{N}{|\Omega|}\right\|^2_{\left(H^1(\Omega)\right)^*}dt\leq C\int_0^\infty \int_\Omega |\nabla w|^2dxdt <\infty, $$
from which it follows
\be
\int_{t_k}^{t_k+1} \left\| w(\cdot,t)-\frac{N}{|\Omega|}\right\|^2_{\left(H^1(\Omega)\right)^*}dt\to 0,\ \ \mbox{ as }k\to\infty.
\label{glo-sta5}
\ee
Combining \eqref{glo-sta4} and \eqref{glo-sta5}, we derive  from the  uniqueness of weak limit that $w_\infty\equiv \frac{N}{|\Omega|}$. While, this is contradictory to \eqref{glo-sta2} and \eqref{glo-sta3}.
\end{proof}

Recall that we are devoted  to the study of the threshold dynamics of \eqref{SIS-1}: when $\mathcal R_0<1$, we know from Theorem \ref{glo-sta-1} that the DFE $(\tilde S,0)$ is globally stable and  it is unstable when $\mathcal R_0>1$ by Proposition \ref{loc-sta}. In the latter case, with the uniform boundedness \eqref{bdd0} at hand, we are going to show  that all the nontrivial solutions of  \eqref{SIS-1} will be attracted by its EE in the case that the rate of disease transmission is proportional to the rate of the disease
recovery, that is,  $\beta(x)=r\gamma (x)$  for some positive constant $r\in (1,\infty)$ and for all $x\in \overline{\Omega}$. In this case, it follows evidently from \eqref{def-R0} that $\mathcal R_0=r$ and so
$$
 r>1\Rightarrow\mathcal R_0>1, \ \ r=1\Rightarrow\mathcal R_0=1 \text{ and } r<1\Rightarrow\mathcal R_0<1.
$$
So far, we have shown  that the unique EE exists if and only if $r > 1$, and $r < 1$ implies DFE is globally stable,  while,  the DFE is neutrally stable for $r = 1$. In the sequel, we shall cope with  the case of $r\geq1$ and aim  to establish  the global attractiveness of EE  for $r>1$ and that of DFE for $r=1$.  If $r>1$, the unique EE exists and is given by
$$
\big(\hat S,\hat I\big)=\left(\frac{1}{r}\frac{ N}{ |\Omega|},\frac{r-1}{r}\frac{N}{|\Omega|}\right).$$
In the case of $r>1$, by constructing a suitable Lyapunov functional, we are able to show the global stability of $(\hat S, \hat I)$ for small \lq\lq chemotactic" sensitivity $\chi$. In the case of $r=1$, upon a careful inspection of the reduced system, the proof of Theorem \ref{glo-sta-1} is adaptable, and so we also have global stability for the unique DFE $(\frac{N}{|\Omega|},0)$.
\begin{theorem}\label{sta-ee}
Assume that $\beta(x)=r\gamma(x)$ for some $r\in [1,\infty)$ and for all $x\in \overline{\Omega}$.
\begin{itemize}
\item[{\rm (i)}] If $r>1,$
then there exists a positive constant $M_0$ depending only on $n,\Omega, \beta,\gamma, d_I$ and $N$ such that whenever $0\leq\chi< \chi_0:=M_0 \sqrt{d_S}$, the unique classical global-in-time  solution $(S,I)$ of \eqref{SIS-1} converges uniformly to the unique EE $\big(\hat S,\hat I\big)$ in the following way:
\be\label{lim-linfty}
\lim_{t\rightarrow \infty}\left(\left\|S(\cdot, t)-\hat S\right\|_{L^\infty(\Omega)}+\left\|I(\cdot, t)-\hat I\right\|_{L^\infty(\Omega)}\right)=0.
\ee
That is, the unique EE $\big(\hat S,\hat I\big)$ of \eqref{SIS-1} is globally stable.
\item[{\rm (ii)}] If $r=1$, then the unique classical global-in-time  solution $(S,I)$ of \eqref{SIS-1} satisfies
 $(S,I)\to \left(\frac{N}{|\Omega|},0\right)$ in $C(\overline\Omega)$ as $t\to\infty.$
That is, the unique DFE $\left(\frac{N}{|\Omega|},0\right)$ of \eqref{SIS-1} is globally stable.
\end{itemize}
\end{theorem}

\begin{proof} (i) We shall use the following Lyapunov functional:
$$V(t):=V(S,I)(t)=\int_\Omega \left[\left(S-\hat S-\hat S\ln \frac{S}{\hat S} \right)+\left(I-\hat I-\hat I\ln \frac{I}{\hat I} \right) \right]dx.$$
Note that, for any $z_0>0$,  the function $f(z)=z-z_0-z_0\ln (\frac{z}{z_0}), z>0$ is strictly decreasing on $(0, z_0)$ and is strictly increasing on $(z_0, \infty)$. Hence, it assumes its global minimum zero at $z=z_0$ and so $V(t)\geq 0$ for all $t\geq 0$ and $V(S,I)=0$ if and only if $(S,I)=(\hat S, \hat I)$.

By \eqref{SIS-1}, we use integration by parts to compute  the time evolution of $V$:
\begin{align}
 \frac{dV(t)}{dt} &=\int_\Omega \frac{S-\hat S}{S}S_tdx+\int_\Omega\frac{I-\hat I}{I}I_tdx \nonumber \\
&  =-d_S\hat S \int_\Omega \frac{|\nabla S|^2}{S^2}dx-\chi\hat S \int_\Omega \frac{\nabla S\cdot \nabla I}{S}dx-d_I\hat I \int_\Omega \frac{|\nabla I|^2}{I^2}dx \nonumber \\
&\ \quad -\int_\Omega \beta(x)I\left(\frac{S}{S+I}-\frac{1}{r}\right)\left(\frac{S-\hat S}{S}-\frac{I-\hat I}{I}\right)dx \nonumber \\
& =-\hat S \int_\Omega \left(\sqrt{d_S}\frac{\nabla S}{S}-\frac{\chi}{2\sqrt{d_S}}\nabla I\right)^2dx- \hat S \int_\Omega \left((r-1)d_I  -\frac{I^2\chi^2}{4d_S}\right)\frac{|\nabla I|^2}{I^2} dx \nonumber \\
& \ \quad -\int_\Omega \frac{\beta(x)SI^2}{(S+I)(\hat S+\hat I)}\left(\frac{\hat I}{I}-\frac{\hat S}{S}\right)^2dx \nonumber \\
&  =- \hat S\int_\Omega \left(d_S - \frac{ I^2\chi^2}{4d_I(r-1)}\right)\frac{|\nabla S|^2}{S^2}dx- \int_\Omega \left(\sqrt{d_I\hat I}\frac{\nabla I}{I}-\frac{\hat{S}I\chi}{2\sqrt{d_I\hat I}\frac{\nabla S}{S}}\right)^2dx \nonumber\\
&\ \quad -\frac{1}{r}\int_\Omega \frac{\beta(x)\hat S}{(S+I)S}\left[(r-1)S-I\right]^2dx,
\label{dVdt}
\end{align}
where we have used the assumption $\beta(x)=r\gamma(x)$ to entail
$$
\frac{1}{r}=\frac{\hat S}{\hat S+\hat I}, \quad \hat I=(r-1)\hat S.
$$
Now, noticing that  $\|I(\cdot,t)\|_{L^\infty}(\leq M_I(n, \Omega, d_I, \gamma_*,\beta^*, S_0, I_0))$ is uniformly bounded  with respect to $\chi$ by \eqref{I-ub}, we  see, if
\be\label{chi-small}
0\leq \chi<\chi_0:=\frac{2}{M_I(n, \Omega, d_I, \gamma_*,\beta^*, S_0, I_0)}\sqrt{(r-1)d_Sd_I},
\ee
then, with the boundedness of $S$ and $I$ as in \eqref{bdd0}, we infer from \eqref{chi-small} and  \eqref{dVdt}  there exists $c_0>0$ such that
\be\label{dVdtf}
\frac{dV(t)}{dt}\leq -c_0\left\{\int_\Omega  |\nabla S|^2 dx+\int_\Omega  |\nabla I|^2  dx+\int_\Omega \Bigr[(r-1)S-I\Bigr]^2dx\right\}.
\ee
Because of $V(t)\geq 0$, an integration of \eqref{dVdtf} from $1$ to $t$ shows
$$
\int_{1}^t\left[\int_\Omega  |\nabla S|^2 dx+\int_\Omega  |\nabla I|^2  dx+\int_\Omega \Bigr[(r-1)S-I\Bigr]^2dx\right]ds\leq \frac{V(1)}{c_0}<\infty, \  \forall t>t_0,
$$
which yields trivially
\be\label{int-finite}
\int_{1}^\infty\left[\int_\Omega  |\nabla S|^2 dx+\int_\Omega  |\nabla I|^2  dx+\int_\Omega \Bigr[(r-1)S-I\Bigr]^2dx\right]ds\leq \frac{V(1)}{c_0}<\infty.
\ee
To proceed further, we need the integrand inside the big square bracket in \eqref{int-finite} to be uniformly bounded and uniformly continuous. For this purpose, we need  further H\"{o}lder type regularity for $S$ and $I$. To achieve this,  we rewrite the $S$-equation  as
 $$
 S_t=\nabla\cdot(A(x,t, \nabla S))+B(x,t)
 $$
 with
 $$
 A(x,t, \xi )=d_S\xi+\chi S(x,t)\nabla I(x,t), \ \ \ \  B(x,t)=-\beta(x) \frac{S(x,t)I(x,t)}{S(x,t)+I(x,t)}+\gamma(x) I(x,t).
 $$
 Then, thanks to the  boundedness information provided in Theorem \ref{glo-sol}, it is an easy matter to check, for some $c_i>0, i=1,2,3$,  that
  \begin{equation*}
  \begin{cases}
  A(x,t,\xi)\cdot \xi\geq \frac{d_S}{2}|\xi|^2-c_1, \ \ \  |A(x,t,\xi)|\leq d_S|\xi|+c_2, \ \  \forall (x,t,\xi)\in\Omega\times (0,\infty)\times \mathbb{R}^n, \\[0.2cm]
  |B(x,t)|\leq c_3, \quad  \forall (x,t,\xi)\in\Omega\times (0,\infty)\times \mathbb{R}^n.
  \end{cases}
  \end{equation*}
 Now,  $S$ and $I$ are bounded in $\Omega\times (0,\infty)$, applying the H\"{o}lder estimates for parabolic  equations (cf. \cite[Theorem 1.3]{PV93-JDE})  and then employing  the  standard parabolic Schauder theory (cf. \cite{Fri-bk, La}), we see  there exist  $\theta\in(0, 1)$ and $C>0$ such that
 \be\label{reg-SI}
 \|S\|_{C^{2+\theta, 1+\frac{\theta}{2}}(\overline{\Omega}\times[t,t+1])}+\|I\|_{C^{2+\theta, 1+\frac{\theta}{2}}(\overline{\Omega}\times[t,t+1])}\leq C, \quad \forall t\geq 1.
 \ee
Thus, the integrand inside the big square bracket in \eqref{int-finite} is uniformly bounded and uniformly continuous, and so \eqref{int-finite} entails
\be\label{lim-t-0}
\lim_{t\rightarrow \infty}\left(\int_\Omega  |\nabla S|^2 dx+\int_\Omega  |\nabla I|^2  dx\right)=0
\ee
and
$$
\lim_{t\rightarrow \infty}\int_\Omega \Bigr[(r-1)S-I\Bigr]^2dx=0.
$$
The latter  along with the conservation of $S+I$ as in \eqref{L1-bdd} gives
\be\label{con-SI}
\lim_{t\rightarrow \infty}\bar{S}:=\lim_{t\rightarrow \infty}\frac{1}{|\Omega|}\int_\Omega   S  dx =\hat S, \ \ \  \  \lim_{t\rightarrow \infty}\bar{I}:=\lim_{t\rightarrow \infty}\frac{1}{|\Omega|}\int_\Omega   I  dx =\hat I.
\ee
 Recalling from the Poincar\'{e} inequality, we have
\begin{align*}
&\int_\Omega \left[\left(S-\hat S\right)^2+\left(I-\hat I\right)^2\right]dx\\
&=\int_\Omega \left[\left(S-\bar S\right)^2 + \left(I-\bar I\right)^2\right]dx+|\Omega|\left[\left(\bar S-\hat S\right)^2+\left(\bar I-\hat I\right)^2\right]\\
&\leq \frac{1}{\lambda_1}\int_\Omega \left( |\nabla S|^2 +  |\nabla I|^2\right)  dx+|\Omega|\left[\left(\bar S-\hat S\right)^2+\left(\bar I-\hat I\right)^2\right].
\end{align*}
This in conjunction with  \eqref{lim-t-0} and \eqref{con-SI} readily shows
 \be\label{lim-l2con}
\lim_{t\rightarrow \infty}\int_\Omega \left[\left(S-\hat S\right)^2+\left(I-\hat I\right)^2\right]dx=0.
 \ee
 Finally,  in view of the Gagliardo-Nirenberg inequality,  we derive from \eqref{reg-SI} that
\begin{align*}
&\ \ \left\|S-\hat S\right\|_{L^\infty(\Omega)}+\left\|I-\hat I\right\|_{L^\infty(\Omega)}\\
&\leq C_{GN}\left(\left\|S-\hat S\right\|_{W^{1,\infty}(\Omega)}^{\frac{n}{n+2}}\left\|S -\hat S \right \|
_{L^2(\Omega)}^{\frac{2}{n+2}}+\left\|I-\hat I \right\|_{W^{1,\infty}(\Omega)}^{\frac{n}{n+2}}\left\|I -\hat I\right\|
_{L^2(\Omega)}^{\frac{2}{n+2}}\right)\\
&\leq C\left(\left\|S-\hat S \right\|
_{L^2(\Omega)}^{\frac{2}{n+2}} +\left\|I-\hat I \right\|
_{L^2(\Omega)}^{\frac{2}{n+2}}\right),
\end{align*}
which coupled with \eqref{lim-l2con} evidently gives rise to \eqref{lim-linfty}.

(ii) Since $\beta(x)=\gamma(x)$, system \eqref{SIS-1} reduces to
\begin{equation}
\label{SIS-1-beta-gamma}
\begin{cases}
\disp S_t=d_S \Delta S+\chi\nabla\cdot( S \nabla I)+\beta(x) \frac{I^2}{S+I},& x\in\Omega, t>0,\\
\disp I_t=d_I\Delta I-\beta(x) \frac{I^2}{S+I},& x\in\Omega, t>0,\\
\disp \frac{\partial S}{\partial \nu}=\frac{\partial I}{\partial \nu}=0, & x\in \partial \Omega, t>0,\\
\disp (S(x,0),  I(x,0))=(S_0(x), I_0(x)), & x\in \Omega.
\end{cases}
\end{equation}
Then the boundedness of $S+I$ in  \eqref{bdd0}  shows that
$$ I_t=d_I\Delta I-\beta(x)\frac{I^2}{S+I}\leq d_I\Delta I-\beta_*\frac{I^2}{M}=d_I\Delta I-\delta I^2,$$
where $\beta_*=\min_{\overline\Omega}\beta$ and $\delta=\beta_*/M>0$.
Now,  we consider the following ODE
\begin{equation}
\left\{ \begin{array}{ll}
\bar I' =-\delta \bar I^2,\ \ \ t>0,\\
\bar I(0)=\|I_0\|_{L^\infty(\Omega)}>0.
\end{array}\right.
\nonumber
\end{equation}
Upon an application of the maximum principle and direct calculations, we have
\begin{equation}
I(\cdot,t )\leq \bar I(t)=\frac{1}{1/\bar I(0)+\delta t}\to 0,\ \ \mbox{ as }t\to\infty.
\label{I-beta-gamma}
\end{equation}
To prove the desired convergence of $S$, one can proceed similarly to the proof of Theorem \ref{glo-sta-1}. In fact, as in \eqref{nabla-I}, we first observe that the algebraic decay \eqref{I-beta-gamma} is sufficient for us to infer
\begin{equation}
\int_0^\infty \int_\Omega |\nabla I|^2dxdt<\infty.
\label{nabla-I-beta-gamma}
\end{equation}
On the other hand, using \eqref{SIS-1-beta-gamma}, we calculate
\begin{align*}
\frac{1}{2}\frac{d}{dt}\int_\Omega S^2dx&=-d_S\int_\Omega |\nabla S|^2dx-\chi \int_\Omega S\nabla I\cdot\nabla Sdx+\int_\Omega\beta\frac{SI^2}{S+I}dx\\
& \leq -d_S\int_\Omega |\nabla S|^2dx+\frac{d_S}{2}\int_\Omega |\nabla S|^2dx+\frac{\chi^2}{2d_S}\int_\Omega S^2|\nabla I|^2dx+\beta^*\int_\Omega I^2 dx\\
& \leq -\frac{d_S}{2}\int_\Omega |\nabla S|^2dx+\frac{\chi^2}{2d_S}\|S(\cdot,t)\|_{L^\infty(\Omega)}^2\int_\Omega |\nabla I|^2dx+\beta^*\int_\Omega I^2 dx.
\end{align*}
Then  \eqref{I-beta-gamma} and \eqref{nabla-I-beta-gamma} enable us to conclude that
\begin{equation}
\int_0^\infty \int_\Omega |\nabla S|^2dxdt<\infty.
\nonumber
\end{equation}
 With these key ingredients obtained, the remaining proof follows along the lines of the argument of Theorem \ref{glo-sta-1}, and hence we omit the details.
\end{proof}

For small $\chi>0$, Theorem \ref{sta-ee} conveys to us that system \eqref{SIS-1} is uniformly persistent when $\mathcal R_0>1$. Equipped with the \lq\lq ultimately uniformly boundedness\rq\rq\, \eqref{bdd},
we can indeed adapt the arguments of \cite[Theorem 3.3]{Peng-Zhao-2012}, developed by Magal and Zhao (see \cite[Theorem 4.5]{MZ} or \cite[Chapter 13]{Zhao-bk}), to deduce that system \eqref{SIS-1} is indeed uniformly persistent for any $\chi>0$. Specifically, we have

\begin{theorem}\label{persist thm} Let $(u_0,v_0)$ obey \eqref{initial} and the basic reproduction number $\mathcal R_0>1$. Then system \eqref{SIS-1} is uniformly persistent, i.e., there exists some $\eta>0$, independent of $(u_0,v_0)$, such that
\be
\nonumber
\liminf_{t\rightarrow \infty} S(x,t)\geq \eta \mbox{ and } \liminf_{t\rightarrow \infty} I(x,t)\geq \eta \text{ uniformly for } x\in\overline{\Omega}.
\ee
Furthermore, there exists at least an EE $({S},{I})$ of \eqref{SIS-1} fulfilling
$$\int_\Omega \left[{S}(x)+{I}(x)\right]dx=N.$$
\end{theorem}

\section{Existence and Uniqueness of EE}

Although Theorem \ref{persist thm} provides us with the existence of EE when $\mathcal R_0>1$, the uniqueness is unclear.
In this section, we shall discuss the existence of EE via a different method. In view of the special reaction terms in system \eqref{SIS-1}, we can reduce the elliptic problem of \eqref{SIS-1} to a single equation, for which the existence of a positive solution can be obtained by a pure PDE approach. Moreover, this technique allows one to deal with the uniqueness and the computations here are also crucial for the forthcoming section where we discuss the asymptotic behavior of EE for small $d_S>0$. Hence, in the following, we focus on the steady state system associated with \eqref{SIS-1}:
\begin{equation}
\left\{ \begin{array}{llll}
\displaystyle d_S\Delta S+\chi \nabla\cdot (S\nabla I)-\beta(x) \frac{SI}{S+I}+\gamma(x) I=0,&x\in\Omega,\\
\displaystyle d_I\Delta I+\beta(x) \frac{SI}{S+I}-\gamma(x) I=0,&x\in\Omega,\\
\displaystyle \frac{\partial S}{\partial \nu}=\frac{\partial I}{\partial \nu}=0,&x\in\partial \Omega,\\
\displaystyle \int_{\Omega}\left[S(x)+I(x)\right]dx=N.
\end{array}\right.
\label{EE-prob}
\end{equation}
Recall that an EE $(S,I)$ is a nonnegative solution of \eqref{EE-prob} with $I\not \equiv0$ on $\Omega$.
A direct application of the strong maximum principle and Hopf boundary point lemma asserts $S,I>0$ on $\overline\Omega.$

Adding the first two  PDEs in \eqref{EE-prob}, we see that
\begin{equation}
\nabla \cdot \left(d_S\nabla S +d_I\nabla I +\chi S \nabla I\right)=0
\nonumber
\end{equation}
or equivalently
\begin{equation}
\nabla \cdot \left[\left(1+\frac{\chi}{d_I}S\right) \nabla \left (\frac{d_S}{\chi} \ln \left(1+\frac{\chi}{d_I}S\right)+I \right)  \right]=0.
\label{SI-1}
\end{equation}
We claim that
\begin{equation}
\frac{d_S}{\chi}\ln \left(1+\frac{\chi}{d_I}S\right)+I \equiv \kappa
\label{SI-2}
\end{equation}
for some positive constant $\kappa$. In fact, upon setting
$$w=\frac{d_S}{\chi}\ln \left(1+\frac{\chi}{d_I}S\right)+I,   $$
we get from  \eqref{SI-1} that
$$\nabla \cdot \left [ \left(1+\frac{\chi }{d_S}S\right)\nabla w \right]=0.$$
As a result,
\begin{align*}
\nabla \cdot \left [ w\left(1+\frac{\chi }{d_S}S\right)\nabla w \right] & =w \nabla \cdot \left [ \left(1+\frac{\chi }{d_S}S\right)\nabla w \right] +\left(1+\frac{\chi}{d_S}S\right)|\nabla w|^2\\
&=\left(1+\frac{\chi}{d_S}S\right)|\nabla w|^2.
\end{align*}
Upon an integration, one sees that $w$ must be constant and hence \eqref{SI-2} holds.

Let
\begin{equation}
\tilde I=\frac{I}{\kappa}\ \ \mbox{and}\ \ \tilde S=\frac{1}{\kappa\chi}\ln \left(1+\frac{\chi}{d_I}S \right).
\label{SI-3}
\end{equation}
Then \eqref{SI-2} gives rise to
\begin{equation}
d_S\tilde S+\tilde I=1
\label{SI-4}
\end{equation}
or
\begin{equation}
\frac{d_S}{\kappa \chi}\ln \left(1+\frac{\chi}{d_I}S\right)+\tilde I=1,
\label{SI-5}
\end{equation}
from which it follows
\begin{equation}
S=\frac{d_I}{\chi}\left [\exp\left\{\frac{\kappa\chi}{d_S}\left(1-\tilde I\right)\right\}-1\right]=:g(\tilde I).
\label{SI-6}
\end{equation}
Define
\begin{equation}
f(x,\tilde I)=\beta(x)\frac{g(\tilde I)}{g(\tilde I)+\kappa \tilde I}-\gamma(x).
\label{SI-7}
\end{equation}
According to the $I$-equation, \eqref{SI-3} and \eqref{SI-6}, it can be easily seen that
$\tilde I$ solves
\begin{equation}
\left\{ \begin{array}{ll}
\displaystyle d_I\Delta \tilde I+\tilde I f(x,\tilde I)=0,&x\in\Omega,\\
\displaystyle \frac{\partial \tilde I}{\partial \nu}=0,&x\in\partial\Omega.
\end{array}\right.
\label{SI-8}
\end{equation}
In addition, the integral constraint  \eqref{EE-prob} and \eqref{SI-3} show
\begin{equation}
N=\int_\Omega \left(S+I\right)dx =\frac{d_I}{\chi}\int_\Omega \left(e^{\kappa\chi \tilde S}-1 \right)dx+\kappa \int_\Omega \tilde Idx.
\label{SI-9}
\end{equation}
These discussions yield equivalent descriptions of the equilibrium problem \eqref{EE-prob}.
\begin{lemma}
\label{equi-lem}
A pair $(S,I)$ is a positive solution of \eqref{EE-prob} if and only if $(\tilde S,\tilde I)$ is a positive solution of \eqref{SI-8}
and \eqref{SI-4} with $\kappa$ being the unique positive constant determined by \eqref{SI-9}.
\end{lemma}

Thanks to the conservation of total population, we have reduced the system \eqref{EE-prob} to a single equation \eqref{SI-8}, and then we can easily establish the existence and uniqueness of EE for the cross-diffusive SIS model \eqref{SIS-1}.

\begin{theorem}
\label{EE-exist}
When  $\mathcal R_0>1$,  the cross-diffusive SIS model \eqref{SIS-1} has a unique EE.
\end{theorem}

\begin{proof}
It is enough to show that \eqref{SI-8} admits a unique positive solution $\tilde I$ with $\tilde I<1$. If so, one can solve $\tilde S \ (>0)$ from \eqref{SI-4}, and then $\kappa>0$ is uniquely determined via \eqref{SI-9}.
The assumption $\mathcal R_0>1$ entails $\lambda^*<0$, where $\lambda^*$ is the principal eigenvalue of the eigenvalue problem \eqref{eigen-prob}. Let $\phi^*>0$ be the corresponding principal eigenfunction. Direct calculations imply that $\underline I=\epsilon \phi^*$ and $\overline I\equiv 1$ is a pair of sub- and super-solutions of \eqref{SI-8}, provided $\epsilon>0$ is chosen to be sufficiently small. Thus, there exists some $\tilde I\in [\underline I, \overline I]$. As a result, $0<\tilde I\leq 1$ on $\overline\Omega$.

We now claim $0<\tilde I <1$ on $\overline\Omega$. In fact, let $\tilde I(x_0)=\max_{\overline\Omega}\tilde I$. Then  the maximum principle \cite[Proposition 2.2]{Lou-Ni} applied to \eqref{SI-8} entails that
$$\beta(x_0)\frac{g(\tilde I(x_0))}{g(\tilde I(x_0))+\kappa \tilde I(x_0)}\geq \gamma(x_0)> 0.$$
By the definition of $g$ in \eqref{SI-6} and the fact that $0<\tilde I\leq 1$, we have $g(\tilde I)\geq 0$. Then the above inequality yields $g(\tilde I(x_0))>0$,  which  in turn shows that  $\tilde I(x)\leq \tilde I(x_0)<1$ for $x\in\overline\Omega.$

Finally, thanks to the fact $\tilde I\in (0,1)$, we simply calculate from \eqref{SI-6} and  \eqref{SI-7}  that  $\frac{\partial f}{\partial \tilde I}(x,\tilde I) <0$ for $x\in \overline\Omega$.  This enables us to deduce the uniqueness of $\tilde I$; see the detailed argument  in the proof of \cite[Lemma 3.3]{Allen}.
\end{proof}

\section{Asymptotic Behavior of EE as $d_S\to 0$}

In this section, we shall study the effect of motility of susceptible population. That is, we will investigate the asymptotic behavior of EE as $d_S\to 0.$ We always assume $\mathcal R_0>1$ so that \eqref{SIS-1} possesses a unique EE by Theorem \ref{EE-exist}. Depending on whether or not  $\beta(x)-\gamma(x)$ changes sign, we consider two different cases.

We first present a simple lemma, providing the asymptotic behavior of $\tilde I$ defined via \eqref{SI-3} for small $d_S>0$.

\begin{lemma}
\label{lem-1}
If $\mathcal R_0>1$, up  to a subsequence of $d_S\to 0$, it holds $\tilde I\to \tilde I^*$ in $C^1(\overline\Omega)$ for some $\tilde I^*\in C^1(\overline\Omega)$ with $0<\tilde I^*(x)\leq 1$ for $x\in\overline\Omega$ and $\partial\tilde I^*/\partial\nu=0$ on $\partial\Omega$.
\end{lemma}

\begin{proof}
First, it follows from  the definition of $f$ in \eqref{SI-7} that
\begin{equation}
\left\| f(x,\tilde I)\right\|_{L^\infty(\Omega)}\leq \beta^*+\gamma^*,\quad \forall d_S>0.
\nonumber
\end{equation}
An application of the Harnack inequality \cite[Lemma 2.2]{Peng-JDE09} to \eqref{SI-8}  gives
\begin{equation}
\max_{\overline\Omega} \tilde I\leq C\min_{\overline\Omega}\tilde I
\label{Har-I tilde}
\end{equation}
for some positive constant $C$ independent of $d_S>0.$

Next,  we have shown that  $0<\tilde I< 1$ and so $\tilde I$ is uniformly bounded for $d_S>0$. Hence, applying the standard $L^p$-estimates to \eqref{SI-8} and the Sobolev embedding theorem, we infer that the $C^{1+\alpha}(\overline\Omega)$-bound of $\tilde I$ is also independent of $d_S>0$ for some $\alpha\in (0,1)$. Thus, after passing to a subsequence of $d_S\to 0$, it holds
\begin{equation}
\tilde I\to \tilde I^*\geq 0\ \mbox{ in } C^1(\overline\Omega),\ \  \mbox{ as } d_S\to 0
\label{Istar-conv}
\end{equation}
for some $\tilde I^*\in C^1(\overline\Omega)$. This $C^1$-convergence enforces that  $\tilde I$ fulfills the homogeneous Neumann boundary condition on $\partial\Omega$.
Furthermore, by  \eqref{Har-I tilde}, it follows
\begin{equation}
\tilde I^*>0\ \mbox{ on }\overline\Omega\quad \mbox{or}\quad \tilde I^*\equiv0.
\nonumber
\end{equation}
On the other hand, it is easy to check from the definitions of  $g$ in \eqref{SI-6} and   $f$ in \eqref{SI-7} that both $g$ and $f$ are decreasing with respect to $d_S>0$. Then the  same argument as \cite[Lemma 4.1]{Allen} shows  that $\tilde I$ is in fact a decreasing function of $d_S$. Consequently, it must hold $\tilde I^*>0$ over $\overline\Omega$. This finishes the proof.
\end{proof}

Now,  we are ready to present the asymptotic behavior of EE  when $\beta>\gamma$ on $\overline\Omega$.
\begin{theorem}
\label{asym-thm1}
Suppose $\beta(x)>\gamma(x)$ on $\overline\Omega$. Then as $d_S\to 0$, any EE $(S,I)$ of the model \eqref{SIS-1} satisfies $(S,I)\to \left(S^*,I^*\right)$ uniformly on $\overline\Omega$ with
\begin{equation}
S^*= \frac{\gamma(x)}{\beta(x)-\gamma(x)}I^* \quad \quad and \quad \quad  I^*=\frac{N}{\int_\Omega \frac{\beta(x)}{\beta(x)-\gamma(x)}dx}.
\label{asym-1}
\end{equation}
\end{theorem}

\begin{proof}
We first notice that $\beta>\gamma$ on $\overline\Omega$ is sufficient to guarantee $\mathcal R_0>1$ and hence a unique EE $(S,I)$ of \eqref{SIS-1} exists by Theorem \ref{EE-exist}. The $I$-equation reads as
\begin{equation}
\left\{ \begin{array}{ll}
\displaystyle d_I\Delta I + \left(\beta\frac{S}{S+I}-\gamma \right)I=0,&x\in\Omega,\\
\displaystyle \frac{\partial I}{\partial \nu}=0,&x\in\partial\Omega.
\end{array}\right.
\nonumber
\end{equation}
Observing
$$\left\|\beta\frac{S}{S+I}-\gamma\right\|_{L^\infty(\Omega)}\leq \beta^*+\gamma^*,\quad \forall d_S>0,$$
we obtain from the Harnack inequality  that
\begin{equation}
\max_{\overline\Omega} I\leq C\min_{\overline\Omega}I
\label{Har-I}
\end{equation}
for some positive constant $C$ independent of $d_S>0$.  Since $\int_\Omega I dx\leq N$, we once again apply  \eqref{Har-I} to end up with
\begin{equation}
\max_{\overline\Omega} I\leq C\min_{\overline\Omega}I\leq \frac{C}{|\Omega|}\int_\Omega I dx\leq \frac{CN}{|\Omega|}.
\nonumber
\end{equation}
Thus, the $L^\infty$-bound of $I$ is independent of $d_S>0$. The same argument leading to  \eqref{Istar-conv} shows, after passing to a subsequence of $d_S\to 0$,  that
\be\label{I-con}
I\to I^*\geq 0\ \mbox{ in } C^1(\overline\Omega),\ \  \mbox{ as } d_S\to 0
\ee
for some $I^*\in C^1(\overline\Omega)$. Moreover, the Harnack inequality \eqref{Har-I} implies
\begin{equation}
I^*>0\ \ \mbox{ on }\overline\Omega\quad \mbox{or}\quad I^*\equiv0 \mbox{ on }\overline\Omega.
\label{I-alt}
\end{equation}
 We now expand out the cross-diffusive term in the $S$-equation and use the $I$-equation to discover that $S$ fulfills
\be
\left\{ \begin{array}{ll}
\displaystyle d_S\Delta S+\chi \nabla S\cdot \nabla I=I\left(1+\frac{\chi S}{d_I}\right) \left(\beta\frac{S}{S+I}-\gamma \right),&x\in\Omega,\\
\displaystyle \frac{\partial S}{\partial I}=0,&x\in\partial\Omega.
\end{array}\right.
\label{S-eqn2}
\ee
Let $S(x_0)=\max_{\overline\Omega}S$.
Then the maximum principle in \cite[Proposition 2.2]{Lou-Ni} entails
\begin{equation}
I(x_0)\left(1+\frac{\chi S(x_0)}{d_I}\right)\frac{(\beta(x_0)-\gamma(x_0))S(x_0)-\gamma(x_0) I(x_0)}{S(x_0)+I(x_0)}\leq 0,
\nonumber
\end{equation}
from which it follows
\begin{equation}
\max_{\overline\Omega}S=S(x_0) \leq \frac{\gamma(x_0)}{\beta(x_0)-\gamma(x_0)} I(x_0) \leq \left(\max_{\overline\Omega}\frac{\gamma}{\beta-\gamma}\right) \|I\|_{L^\infty(\Omega)}.
\nonumber
\end{equation}
This in  conjunction with \eqref{I-alt} and \eqref{I-con} forces $I^*>0$, since otherwise both $S$ and $I$ are small for sufficiently small $d_S>0$, contradicting the prescribed mass conservation $\int_\Omega [S(x)+I(x)]dx=N$.

We now claim that $\tilde I^*\equiv 1$, where $\tilde I^*$ is given in Lemma \ref{lem-1}. Suppose not,  then according to Lemma \ref{lem-1}, we have $$\int_\Omega \left(1-\tilde I^*\right)dx>0.$$
In view of the relation of $S$ and $\tilde I$ in \eqref{SI-5}, for small $d_S>0$, we deduce
\begin{align*}
0<\frac{1}{2}\int_{\Omega} \left(1-\tilde I^*\right)dx<\int_\Omega \left(1-\tilde I\right)dx&=\frac{d_S}{\kappa \chi}\int_\Omega \ln\left(1+\frac{\chi}{d_I}S\right)dx\\
&\leq
\frac{d_S}{\kappa \chi}\int_\Omega\frac{\chi}{d_I}Sdx\leq \frac{d_S}{\kappa d_I}N,
\end{align*}
from which it follows that $\kappa\to 0$ as $d_S\to 0$. As a result, $I\to 0$ as $d_S\to 0$ due to \eqref{SI-2}. However, this is a contradiction to $I^*>0$, as we have just proved. Thus, we must have $\tilde I^*\equiv 1$. Consequently, since $I\to I^*$ in $C^1(\overline\Omega)$ as $d_S\to 0$ and $I=\kappa \tilde I$ for positive constant $\kappa$, then $I^*$ is necessarily a positive constant.

Since $I\to I^*\equiv\mbox{const}>0$ in $C^1(\overline\Omega)$ (in turn $|\nabla I|\to0$ uniformly on $\overline\Omega$) as $d_S\to 0$, a standard singular perturbation argument (see, for instance, \cite[Lemma 2.4]{DPW} or \cite{HLM}) applied to \eqref{S-eqn2} yields that
 \begin{eqnarray}
 S(x)\to S^*(x)=\frac{\gamma(x)}{\beta(x)-\gamma(x)}I^*\ \ \mbox{ uniformly on }\overline\Omega,
 \label{singular}
\end{eqnarray}
as $d_S\to 0$. Then the conservation of total population
$$\int_{\Omega}\left(S^*+I^*\right)dx=N$$
simply gives \eqref{asym-1}. The uniqueness of $(S^*,I^*)$ says that all the above limits hold without passing to a subsequence of $d_S\to 0$.
\end{proof}

In the sequel, besides  $\mathcal R_0>1$, we shall assume that the set $\{x\in\overline\Omega:\beta(x)<\gamma(x)\}$ is nonempty, which in fact indicates that $\beta(x)-\gamma(x)$ must change sign. Note that all of $S$, $I$ and $\kappa>0$ in \eqref{SI-2} depend on $d_S$. Hence, to determine their asymptotics as $d_S\to 0$, we shall start with the limiting function of $\tilde I$, i.e., $\tilde I^*\in (0,1]$.
To further study the limiting function $\tilde I^*$, we need to determine where $0<\tilde I^*<1$ and where $\tilde I^*=1$. For these purposes, we define
\begin{equation}
H^-=\{x\in\Omega: \beta(x)<\gamma(x)\}\quad \mbox{and}\quad H^+=\{x\in\Omega: \beta(x)>\gamma(x)\}
\nonumber
\end{equation}
and
\begin{equation}
J^-=\{x\in\overline\Omega: 0<\tilde I^*(x)<1\}\quad \mbox{and}\quad J^+=\{x\in\overline\Omega: \tilde I^*(x)=1\}.
\label{sets-2}
\end{equation}

\begin{lemma}
Suppose that $\mathcal R_0>1$ and $\{x\in\overline\Omega:\beta(x)<\gamma(x)\}\neq \emptyset$.
Then the following statements hold.

\begin{itemize}
\item[{\rm (i)}] $H^- \subset J^-$;
\item[{\rm (ii)}] after passing to a subsequence of $d_S\to 0$, $\kappa\to 0$ and $I\to 0$ in $C^1(\overline\Omega)$;
\item[{\rm (iii)}] after passing to a subsequence of $d_S\to 0$, $\kappa/d_S\to M>0$, where $M$ is the unique number satisfying
\begin{equation}
\int_\Omega e^{\chi (1-\tilde I^*(x))M}dx=\frac{N\chi}{d_I}+|\Omega|
\label{M}
\end{equation}
and
\begin{equation}
S(x)\to S^*(x):=\frac{d_I}{\chi} \left[e^{\chi(1-\tilde I^*(x))M }-1 \right] \mbox{    in    }  C^1(\overline\Omega).
\label{S-star}
\end{equation}

\end{itemize}

\end{lemma}

\begin{proof}
(i) can be proved by using an indirect argument as  in \cite[Lemma 4.3]{Allen}.

(ii) Firstly, it  follows from \eqref{SI-5} that
\begin{equation}
\frac{\kappa\chi}{d_S}\int_\Omega \left(1-\tilde I\right)dx=\int_\Omega \ln \left(1+\frac{\chi}{d_I}S\right)dx\leq \int_\Omega \frac{\chi}{d_I}Sdx\leq \frac{\chi}{d_I}N.
\label{lem2-1}
\end{equation}
On the other hand, as $d_S\to 0$, the item (i) implies
$$\int_\Omega \left(1-\tilde I\right)dx\to  \int_\Omega \left(1-\tilde I^*\right)dx\geq \int_{H^-} \left(1-\tilde I^*\right)dx>0. $$
As a result, for small $d_S>0$, it holds
$$\int_\Omega \left(1-\tilde I\right)dx\geq \frac{1}{2}\int_{H^-}\left(1-\tilde I^*\right)dx>0.$$
This, together with \eqref{lem2-1},  indicates that for small $d_S>0$,
$$
0<\frac{\kappa}{2d_S}\int_{H^-} \left(1-\tilde I^*\right)dx\leq \frac{N}{d_I}.
$$
This forces $\kappa\to 0$ as $d_S\to 0$, and so  \eqref{SI-2} ensures $I\leq \kappa \to 0$ uniformly on $\overline\Omega$ as $d_S\to 0$. Since $\tilde I\to \tilde I^*$ in $C^1(\overline\Omega)$ as $d_S\to 0$ by Lemma \ref{lem-1}, it follows from $I=\kappa \tilde I$ that $I \to 0$ in $C^1(\overline\Omega)$ as $d_S\to 0$.

(iii) Using \eqref{SI-6}, the fact that $\tilde I\to \tilde I^*$ and (i), we get
\begin{equation}
\begin{split}
N=\int_\Omega Idx+\int_\Omega Sdx&=\int_\Omega Idx +\frac{d_I}{\chi} \int_\Omega \left[e^{\frac{\kappa\chi}{d_S}(1-\tilde I)}-1 \right]dx \\
&\geq \frac{d_I}{2\chi}\int_{H^-}\left[e^{\frac{\kappa\chi}{d_S}(1-\tilde I^*)}-1 \right]dx>0
\end{split}
\label{lem2-2}
\end{equation}
for small $d_S>0$. This first  tells us that $\kappa/d_S$ is bounded for small $d_S>0$. Thus, after further passing to a subsequence of $d_S\to 0$ if necessary, we can assume that $\kappa/d_S\to M\geq 0$. Moreover, by sending $d_S\to 0$ in \eqref{lem2-2} and using (ii), we see that $M$ is determined via \eqref{M}. Obviously, such $M$ is unique and $M>0$. Finally, \eqref{S-star} can be seen from \eqref{SI-6} and Lemma \ref{lem-1}.
\end{proof}

\begin{remark}
Our results here match those of \cite[Lemma 4.5]{Allen} with $\chi=0$. In fact, if formally letting $\chi\to 0$ in \eqref{SI-5}, we obtain that
$$\frac{d_S S}{\kappa d_I}+\tilde I=1.$$
Upon an integration and using the conservative property of total population, we have
$$\frac{d_S}{\kappa d_I}\left(N-\int_\Omega Idx\right)=\int_\Omega(1-\tilde I)dx.$$
Since $I\to 0$ and $\tilde I\to \tilde I^*$ as $d_S\to 0$, it follows that
$$\frac{\kappa }{d_S}\to N^*:=\frac{N}{d_I\int_\Omega(1-\tilde I^*)dx}.$$
Furthermore, if we let $\chi\to 0$ in \eqref{S-star}, we formally obtain that
$$S\to S^*=d_I(1-\tilde I^*)N^*. $$
Those are the asympotics proved in  \cite[Lemma 4.5]{Allen}.
\end{remark}

In view of \eqref{sets-2} and \eqref{S-star}, it holds
\begin{equation}
J^-=\{x\in\overline\Omega:S^*(x)>0\} \quad \mbox{and}\quad J^+=\{x\in\overline\Omega:S^*(x)=0\}.
\nonumber
\end{equation}

Using the same arguments as those in \cite[Lemmas 4.6 and 4.7]{Allen}, one can further prove the following properties of the sets $J^+$ and $J^-$.

\begin{lemma}
Suppose that $\mathcal R_0>1$ and $\{x\in\overline\Omega:\beta(x)<\gamma(x)\}\neq \emptyset$. Then
$\emptyset\neq J^+\subset \overline{ H^+}$ and the set $J^+$ has positive measure. If we further assume that the set $H^0=\{x\in\Omega:\beta(x)=\gamma(x)\}$ consists of finitely many disjoint $C^1$-surfaces (or finitely many points if $n=1$, each of which is a simple root of $\beta-\gamma$). Then $\overline {H^-}$ is a proper subset of $J^-$. Moreover, $\tilde I^*\in C^2(J^-)$ satisfies $d_I\Delta \tilde I^*+(\beta-\gamma)\tilde I^*=0$ on $J^-$.
\end{lemma}

We summarize the findings above in the following theorem.

\begin{theorem}
Suppose that $\mathcal R_0>1$ and $\{x\in\overline\Omega:\beta(x)<\gamma(x)\}\neq \emptyset$.

\begin{itemize}
\item[{\rm (i)}] As $d_S\to 0$, any EE $(S,I)$ of the cross-diffusive SIS model \eqref{SIS-1} satisfies $(S,I)\to \left(S^*,0\right)$ in $C^1(\overline\Omega)$ with $S^*$ satisfying
\begin{equation}
\nonumber
S^*\geq 0,\quad \frac{\partial S^*}{\partial \nu}=0 \mbox{ on } \partial\Omega\quad  \mbox{ and } \quad  \int_\Omega S^*(x)dx=N.
\end{equation}
\item[{\rm (ii)}] The set $J^-:=\{x\in\overline\Omega:S^*(x)>0\}$ contains $H^-$;

\item[{\rm (iii)}] The set $J^+:=\{x\in\overline\Omega:S^*(x)=0\}$ has positive measure and it is contained in $\overline {H^+}$;

\item[{\rm (iv)}] If we further assume that the set $H^0=\{x\in\Omega:\beta(x)=\gamma(x)\}$ consists of finitely many disjoint $C^1$-surfaces (or finitely many points if $n=1$, each of which is a simple root of $\beta-\gamma$), then $\overline{H^-}\subset J^-$ and the set $J^-\setminus H^-$ has positive measure.

\end{itemize}

\end{theorem}

\section{The Model with Varying Total Population}

In this section, we briefly discuss the model \eqref{SIS-2} with cross-diffusion and linear source.  We first notice that the global existence and boundedness of solutions have been included in  Theorem \ref{glo-sol}. Next,  we plan to discuss the global stability of the equilibria of system \eqref{SIS-2}.
It can be easily seen that the unique DFE is given by $(\tilde S,0)$, where $\tilde S$ is the unique positive solution of
\begin{equation}
\left\{ \begin{array}{ll}
\displaystyle d_S \Delta S+\Lambda(x)-S=0,&x\in\Omega,\\
\displaystyle \frac{\partial S}{\partial \nu}=0,&x\in\partial\Omega.
\end{array}\right.
\nonumber
\end{equation}

Similar to Theorem \ref{glo-sta-1}, we have
\begin{theorem}
\label{glo-sta-2}
If the basic reproduction number $\mathcal R_0<1$, then the unique global-in-time  bounded and classical solution $(S, I)$  of the cross-diffusive SIS model \eqref{SIS-2} satisfies $(S,I)\to (\tilde S,0)$ in $C(\overline\Omega)$ as $t\to\infty.$
\end{theorem}

\begin{proof}
 Let  $v(x,t)=S(x,t)-\tilde S(x)+I(x,t)$. Then straightforward calculations from \eqref{SIS-2}  show that $v$ satisfies
\begin{equation}
\left\{ \begin{array}{lll}
\displaystyle v_t=d_S\Delta v+\nabla \cdot \left( (d_I-d_S)\nabla I+\chi S\nabla I \right)-v+I,&x\in\Omega, t>0,\\
\displaystyle \frac{\partial v}{\partial \nu}=0,&x\in\partial \Omega, t>0,\\
v(x,0)=S_0(x)-\tilde S(x)+I_0(x),&x\in\Omega.
\end{array}\right.
\label{v}
\end{equation}
We multiply the first equation in \eqref{v} by $v$, integrate by parts and employ the Cauchy-Schwarz inequality to obtain
\begin{align*}
\frac{1}{2}\frac{d}{dt}\int_\Omega v^2dx&=-d_S\int_\Omega |\nabla v|^2dx-(d_I-d_S)\int_\Omega \nabla I\cdot\nabla vdx\\
&\ \ \ -\chi \int_\Omega S\nabla I\cdot\nabla vdx-\int_\Omega v^2dx+\int_\Omega Ivdx\\
 &\leq -d_S\int_\Omega |\nabla v|^2dx+\frac{d_S}{4}\int_\Omega |\nabla v|^2dx+\frac{(d_I-d_S)^2}{d_S}\int_\Omega |\nabla I|^2dx\\
&\ \ +\frac{d_S}{4}\int_\Omega |\nabla v|^2dx+\frac{\chi^2}{d_S}\int_\Omega S^2|\nabla I|^2dx- \int_\Omega v^2dx+\frac{1}{2}\int_\Omega v^2dx+\frac{1}{2}\int_\Omega I^2dx \\
&\leq -\frac{d_S}{2}\int_\Omega |\nabla v|^2dx+ \left[\frac{(d_I-d_S)^2}{d_S}+\frac{\chi^2\|S(\cdot,t)\|^2_{L^\infty(\Omega)}}{d_S} \right]\int_\Omega |\nabla I|^2dx+\frac{1}{2}\int_\Omega I^2dx.
\end{align*}
 The fact that $\mathcal R_0<1$ shows that \eqref{I-decay} and \eqref{nabla-I} remain valid.  Then upon an integration in the time  variable and in view of \eqref{I-decay} and \eqref{nabla-I}, we find that
\begin{equation}
\int_0^\infty \int_\Omega |\nabla v|^2dxdt<\infty.
\label{thm6-1}
\end{equation}
 Next, we shall prove a weak stabilization of $v_t$, an analog of Lemma \ref{wt-decay}. To this purpose, we first integrate by \textcolor{red} {parts} both sides of the PDE in \eqref{v}, we end up with
\begin{equation}
\frac{d}{dt}\int_\Omega v(x,t)dx=-\int_\Omega v(x,t)dx+\int_\Omega I(x,t)dx.
\label{thm6-2}
\end{equation}
Thanks to the exponential decay of $I$ in \eqref{I-decay}, we can easily infer from \eqref{thm6-2} that $\int_\Omega v(\cdot,t)$  decays exponentially to zero, and hence, its average  $\bar v(t):=\frac{1}{|\Omega|}\int_\Omega v(\cdot,t)$ decays also exponentially to zero as $t\to\infty.$ According to the Poincar\'{e} inequality, there exists some generic positive constant $C$ such that
\begin{equation*}
\|v(\cdot,t)\|_{L^2(\Omega)}\leq \|v(\cdot,t)-\bar v(t)\|_{L^2(\Omega)}+ \|\bar v(t)\|_{L^2(\Omega)}
 \leq C\|\nabla v(\cdot,t)\|_{L^2(\Omega)} +\|\bar v(t)\|_{L^2(\Omega)}.
\end{equation*}
By the exponential decay of $\bar v(t)$ and \eqref{thm6-1}, it then follows that
\begin{equation}
\int_0^\infty \|v(\cdot,t)\|^2_{L^2(\Omega)}dt \leq 2C^2 \int_0^\infty \|\nabla v(\cdot,t)\|^2_{L^2(\Omega)}dt+2|\Omega|\int_0^\infty |\bar v(t)|^2 dt <\infty.
\label{thm6-3}
\end{equation}
Now,  for any test function $\varphi\in H^1(\Omega)$, we deduce from \eqref{v} that
\be
\nonumber
\begin{split}
\int_\Omega v_t\varphi dx&=-d_S \int_\Omega \nabla v\cdot\nabla \varphi dx -\int_\Omega v\varphi dx+\int_\Omega I\varphi dx \\
&\ \ \ -\int_\Omega \left[(d_I-d_S)\nabla I+\chi S\nabla I\right]\cdot \nabla \varphi dx\\
&\leq d_S \|\nabla v(\cdot,t)\|_{L^2(\Omega)}  \|\nabla \varphi \|_{L^2(\Omega)}+ \left(\|v(\cdot,t)\|_{L^2(\Omega)}+\|I(\cdot,t)\|_{L^2(\Omega)} \right) \|\varphi\|_{L^2(\Omega)} \\
&\quad + \left[(d_I+d_S) \|\nabla I(\cdot,t)\|_{L^2(\Omega)}+\chi \|S(\cdot,t)\|_{L^\infty(\Omega)} \|\nabla I(\cdot,t)\|_{L^2(\Omega)} \right] \|\nabla \varphi \|_{L^2(\Omega)}\\
&\leq \big[d_S \|\nabla v(\cdot,t)\|_{L^2(\Omega)} +\left(d_I+d_S+ \chi \|S(\cdot,t)\|_{L^\infty(\Omega)}\right) \|\nabla I(\cdot,t)\|_{L^2(\Omega)} \\
&\quad +\|v(\cdot,t)\|_{L^2(\Omega)}+\|I(\cdot,t)\|_{L^2(\Omega)}\big]
\| \varphi \|_{H^1(\Omega)},
\end{split}
\ee
which gives
\begin{align*}
\left\| v_t(\cdot,t)\right\|_{\left(H^1(\Omega)\right)^*}&\leq d_S \|\nabla v(\cdot,t)\|_{L^2(\Omega)}+\left(d_I+d_S+ \chi \|S(\cdot,t)\|_{L^\infty(\Omega)}\right) \|\nabla I(\cdot,t)\|_{L^2(\Omega)}\\
&\quad +\|v(\cdot,t)\|_{L^2(\Omega)}+\|I(\cdot,t)\|_{L^2(\Omega)}.
\end{align*}
Combining this with  \eqref{I-decay}, \eqref{nabla-I}, \eqref{thm6-1} and \eqref{thm6-3}, we derive that
\begin{equation}
\int_0^\infty \left\| v_t(\cdot,t)\right\|_{\left(H^1(\Omega)\right)^*}^2dt<\infty.
\label{thm6-4}
\end{equation}
With the help of \eqref{thm6-1} and \eqref{thm6-4}, combined with the exponential decay of $\int_\Omega v(\cdot,t)$, using an argument similar to that  of Theorem \ref{glo-sta-1}, one can readily show that in fact $v(\cdot,t)\to 0$ in $C(\overline\Omega)$ as $t\to\infty$. Since we have already known $I(\cdot,t)\to 0$ uniformly, then it follows $S(\cdot,t)=v(\cdot,t)+\tilde S-I(\cdot,t)\to \tilde S$ uniformly on $\overline\Omega$ as $t\to\infty.$
\end{proof}

In the case $\mathcal R_0>1$, we have
\begin{theorem}
\label{sta-ee-2}
Suppose that all of $\beta$, $\gamma$ and $\Lambda$ are all positive constants and that $\beta>\gamma$ (so that $\mathcal R_0=\beta/\gamma>1$). Then there exists a positive constant $M_1$ depending only on $n,\Omega, \beta,\gamma, d_I$ and $\hat N$ such that whenever $0\leq\chi< \chi_0:=M_1 \sqrt{d_S}$, the unique classical global-in-time  solution $(S,I)$ of \eqref{SIS-2} converges uniformly to the unique EE $(\hat S,\hat I)$ in the following fashion:
\be
\nonumber
\lim_{t\rightarrow \infty}\left(\left\|S(\cdot, t)-\hat S\right\|_{L^\infty(\Omega)}+\left\|I(\cdot, t)-\hat I\right\|_{L^\infty(\Omega)}\right)=0,
\ee
where
$$(\hat S,\hat I)=\left(\Lambda,\frac{\beta-\gamma}{\gamma}\Lambda\right).$$
\end{theorem}

\begin{proof}
The proof is very similar to that of Theorem \ref{sta-ee} (i) by using the same Lyapunov functional; the details are thus omitted here.
\end{proof}

In the general situation, as in proving Theorem \ref{persist thm} for the system \eqref{SIS-1},
we can employ the abstract dynamical systems theory to conclude the uniform persistence property for the system \eqref{SIS-2}.
That is, we can state

\begin{theorem}\label{persist thm2} Let $(u_0,v_0)$ obey \eqref{initial} and $\mathcal R_0>1$. Then system \eqref{SIS-2} is uniformly persistent, i.e., there exists some $\eta>0$, independent of $(u_0,v_0)$, such that
\be
\nonumber
\liminf_{t\rightarrow \infty} S(x,t)\geq \eta \mbox{ and } \liminf_{t\rightarrow \infty} I(x,t)\geq \eta \text{ uniformly for } x\in\overline{\Omega}.
\ee
Furthermore, \eqref{SIS-2} admits at least an EE $({S},{I})$.
\end{theorem}

Hence, when $\mathcal R_0>1$, Theorem \ref{persist thm2} ensures the existence of EE to \eqref{SIS-2}, which solves the following elliptic problem:
\begin{equation}
\left\{ \begin{array}{lll}
\displaystyle d_S\Delta S+\chi \nabla\cdot (S\nabla I)-\beta(x) \frac{SI}{S+I}+\gamma I+\Lambda(x) - S =0,&x\in\Omega,\\
\displaystyle d_I\Delta I+\beta(x) \frac{SI}{S+I}-\gamma(x) I=0,&x\in\Omega,\\
\displaystyle \frac{\partial S}{\partial \nu}=\frac{\partial I}{\nu}=0,&x\in\partial \Omega.
\end{array}\right.
\label{EE-prob2}
\end{equation}
For \eqref{EE-prob2}, we only capture   the following information about the asymptotic profile of EE of \eqref{SIS-2}  as $d_S\to 0$, which is  poorer  than that of \eqref{EE-prob}.
\begin{theorem}
\label{asy-ee-prob2}
Assume that $\mathcal R_0>1$. Fix $d_I>0$ and let $d_S\to 0$, then every positive solution $(S,I)$ of the problem  \eqref{EE-prob2} satisfies (up to a subsequence of $d_S\to 0$)
$$I\to I^*\ \ \mbox{in}\ C^1(\overline\Omega)$$
for some positive function $I^*\in C^1(\overline\Omega)$.
\end{theorem}

\begin{proof}
We integrate by parts from \eqref{EE-prob2} to obtain
\begin{equation}
-\int_\Omega \beta \frac{SI}{S+I}dx+\int_\Omega \gamma Idx+\int_\Omega \Lambda dx-\int_\Omega Sdx=0,
\nonumber
\end{equation}
and
\begin{equation}
\int_\Omega \beta\frac{SI}{S+I}dx-\int_\Omega\gamma Idx=0.
\label{thm7-1}
\end{equation}
Adding the two identities, we see
$$\int_\Omega S dx=\int_\Omega \Lambda dx.$$
Then  \eqref{thm7-1} moreover gives
\begin{eqnarray}
\gamma_* \int_\Omega I dx\leq \int_\Omega \gamma I dx=\int_\Omega \beta\frac{SI}{S+I}dx\leq \beta^* \int_\Omega S dx = \beta^* \int_\Omega \Lambda dx.
\label{thm7-2}
\end{eqnarray}
Thus, the $L^1$-norm of $S$ and $I$ is uniformly bounded  with respect to $d_S>0$.

Now,  thanks to \eqref{thm7-2}, as in the proof of Theorem \ref{asym-thm1},
after passing to a subsequence of $d_S\to 0$, it holds
\begin{equation}
I\to I^*\geq 0\ \mbox{ in } C^1(\overline\Omega),\ \  \mbox{ as } d_S\to 0,
\nonumber
\end{equation}
for some $I^*\in C^1(\overline\Omega)$.
Furthermore, we have the dichotomy:
\begin{equation}
I^*>0\ \mbox{ on }\overline\Omega\quad \mbox{or}\quad I^*\equiv0.
\nonumber
\end{equation}
Assume that $I^*\equiv 0$; that is, $I\to 0$ in $C^1(\overline\Omega)$ as $d_S\to 0$. Expanding out the cross-diffusive term in the first equation of \eqref{EE-prob2} and using the second equation, we see that $S$ satisfies
\begin{equation}
\left\{ \begin{array}{ll}
\displaystyle d_S\Delta S+\chi \nabla I\cdot \nabla S+\left(\gamma-\frac{\beta S}{S+I}\right) \left(1+\frac{\chi S}{d_I}\right)I+\Lambda -S=0,&x\in\Omega,\\
\displaystyle \frac{\partial S}{\partial \nu}=0,&x\in\partial\Omega.
\end{array}\right.
\label{thm7-3}
\end{equation}
Then, since $I\to 0$ in $C^1(\overline\Omega)$ as $d_S\to 0$, as in the proof of \eqref{singular}, a singular perturbation argument applied to \eqref{thm7-3} yields that
$$
S\to \Lambda \quad \mbox{uniformly on } \overline\Omega,\ \ \ \mbox{as}\ d_S\to0.
$$

Now, to proceed, we  let
$$h=\frac{I}{\|I\|_{L^\infty(\Omega)}}>0.$$
Then $\|h\|_{L^\infty(\Omega)}=1$ and it can be easily seen that $h$ satisfies
\begin{equation}
\left\{ \begin{array}{ll}
\displaystyle d_I\Delta h+\left(\beta\frac{S}{S+I}-\gamma\right)h=0,&x\in\Omega,\\
\displaystyle \frac{\partial h}{\partial \nu}=0,&x\in\partial\Omega.
\end{array}\right.
\label{thm7-4}
\end{equation}
The Harnack inequality \cite[Lemma 2.2]{Peng-JDE09} tells us that
\begin{equation}
\max_{\overline\Omega} h\leq C\min_{\overline\Omega}h
\label{thm7-5}
\end{equation}
for some positive constant $C$ independent of $d_S>0$.
Since $h$ is uniformly bounded by $1$, after passing to a subsequence of $d_S\to 0$, we have, as before,
\begin{equation}
h\to \tilde h\quad \mbox{in } C^1(\overline\Omega),
\nonumber
\end{equation}
for some $0\leq\tilde h\in C^1(\overline\Omega)$ with $\|\tilde h\|_{L^\infty(\Omega)}=1.$ In light  of the Harnack inequality \eqref{thm7-5},
it is necessary that $\tilde h>0$ on $\overline\Omega$. On the other hand, since $S\to \Lambda>0$ and $I\to 0$ uniformly on $\overline\Omega$ as $d_S\to 0$, it follows from \eqref{thm7-4} that $\tilde h$ satisfies
\begin{equation}
\left\{ \begin{array}{ll}
\displaystyle d_I\Delta\tilde h+\left(\beta-\gamma\right)\tilde h=0,&x\in\Omega,\\
\displaystyle \frac{\partial\tilde h}{\partial \nu}=0,&x\in\partial\Omega.
\end{array}\right.
\nonumber
\end{equation}
This says that the principal eigenvalue $\lambda^*$ of the eigenvalue problem \eqref{eigen-prob} is zero, contradicting $\mathcal R_0>1$ and the fact that $1-\mathcal R_0$ and $\lambda^*$ have the same sign.
Therefore, it holds
\begin{equation}
I\to I^*>0\ \mbox{ in } C^1(\overline\Omega),\ \ \mbox{ as } d_S\to 0.
\nonumber
\end{equation}
This finishes the proof.
\end{proof}

\textbf{Acknowledgments.}  We would like to thank the two anonymous referees for their careful reading of our manuscript, and constructive comments and suggestions, which further help us to improve the presentation.


\end{document}